\documentclass{amsart}
\usepackage[utf8]{inputenc}
\usepackage[T1]{fontenc}
\usepackage{amsmath, amsthm, amssymb, amsfonts, mathabx}
\usepackage{natbib}

\usepackage{tikz}
\usetikzlibrary{calc,decorations,decorations.pathreplacing,decorations.shapes,decorations.markings,shapes,patterns}

\usepackage{geometry}
\usepackage{pdflscape}

\usepackage{hyperref}
\newcommand{\oeis}[1]{\href{https://oeis.org/#1}{#1}}

\usepackage{tablefootnote}

\usepackage{xargs}

\newcommand{\shadetheboxesPM}[1]{
    \foreach \x/\y in {#1}
    \fill[pattern color = black!75, pattern=north east lines] (\x,\y) rectangle +(1,1);
}

\newcommand{\drawthegridstrule}[1]{
    \foreach \x/\y in {#1}
    \draw[line width=2] (0,0) grid (\x,\y);
}

\newcommand{\drawverticallines}[3]{
    \foreach \x in {#2}
    \draw[line width=#3] (\x+0.01,0.01) -- (\x+0.01,#1+0.99);
}

\newcommand{\drawhorizontallines}[3]{
    \foreach \y in {#2}
    \draw[line width=#3] (0.01,\y+0.01) -- (#1+0.99,\y+0.01);
}

\newcommand{\drawclpattern}[2]{
	\foreach[count=\x] \y in {#1}
	{
		\filldraw (\x,\y) circle (#2 pt);
	}
}

\newcommand{\drawspecialbox}[1]{
    \foreach \x/\y/\z/\w/\A in {#1}
    {
        \fill[color = white!100, opacity=1, rounded corners = 1.5pt] (\x+0.125,\y+0.125) rectangle (\z-0.125,\w-0.125);
        \draw[color = black, rounded corners = 1.5pt] (\x+0.125,\y+0.125) rectangle (\z-0.125,\w-0.125);
        \fill[black] (\x/2+\z/2,\y/2+\w/2) node {\A};
    }
}

\newcommand{\drawspecialboxcorner}[1]{
    \foreach \x/\y/\A in {#1}
    {
        \fill[color = white!100, opacity=1, rounded corners = 1.5pt] (\x+0.125,\y+0.125) rectangle (\x+1-0.125,\y+1-0.125);
        \draw[color = black, rounded corners = 1.5pt] (\x+0.125,\y+0.125) rectangle (\x+1-0.125,\y+1-0.125);
        \fill[black] (\x+1/2,\y+1/2) node {\A};
    }
}

\newcommand{\drawspecialboxlarge}[1]{
    \foreach \x/\y/\z/\w/\A in {#1}
    {
        \fill[color = white!100, opacity=1, rounded corners = 1.5pt] (\x+0.125,\y+0.125) rectangle (\z-0.125,\w-0.125);
        \draw[color = black, rounded corners = 1.5pt] (\x+0.125,\y+0.125) rectangle (\z-0.125,\w-0.125);
        \fill[black] (\x/2+\z/2,\y/2+\w/2) node {\Large \A};
    }
}

\newcommand{\drawsolidshadedbox}[1]{
    \foreach \x/\y/\z/\w/\A in {#1}
    {
        \fill[color = gray!50, opacity=1, rounded corners=1.5pt] (\x+0.125,\y+0.125) rectangle (\z-0.125,\w-0.125);
        \draw[color = black, rounded corners=1.5pt] (\x+0.125,\y+0.125) rectangle (\z-0.125,\w-0.125);
        \fill[black] (\x/2+\z/2,\y/2+\w/2) node {\A};
    }
}

\newcommand{\drawlabels}[1]{
	\foreach \x/\y/\lab in {#1}
	{
		\draw (\x + 0.5,\y + 0.5) node {\lab};
	}
}

\newcommandx{\patt}[9][4={},5={},6={},7={},8={},9=4]
{
	\scalebox{#1}
	{
		\begin{tikzpicture}[baseline=(current bounding box.center)]
			\useasboundingbox (0.0,-.3) rectangle (#2+1,#2+1.3);
			\shadetheboxesPM{#4}
			\draw (0.01,0.01) grid (#2+1-0.01,#2+1-0.01);

			\drawsolidshadedbox{#6}
			\drawspecialbox{#7}
			\drawspecialboxlarge{#5}
			\drawclpattern{#3}{#9}
			\drawlabels{#8}
		\end{tikzpicture}
	}
}

\newcommandx{\cpatt}[8][4={},5={},6={},7={},8={}]
{
	\scalebox{#1}
	{
		\begin{tikzpicture}[baseline=(current bounding box.center)]
			\useasboundingbox (0.0,-.3) rectangle (#2+1,#2+1.3);
			\shadetheboxesPM{#4}
			\draw (0.01,0.01) grid (#2+1-0.01,#2+1-0.01);

			\drawsolidshadedbox{#6}
			\drawspecialbox{#7}
			\drawspecialboxlarge{#5}
			\drawclpattern{#3}{4}

			\foreach \x/\y in {#8}
			{
				\draw[line width=1] (\x,\y) circle (7 pt);
			}
		\end{tikzpicture}
	}
}

\newcommandx{\metapatt}[8][6={},7={},8={}]
{
    \scalebox{#1}
    {
        \begin{tikzpicture}[baseline=(current bounding box.center)]
					\foreach \width/\height in {#2}
					{
						\useasboundingbox (0.0,-.3) rectangle (\width+1,\height+1.3);
            \shadetheboxesPM{#6}

            \foreach \pos/\type in {#4}
            {
                \ifthenelse{\equal{\type}{v}}
                {
                    \drawverticallines{\height}{\pos}{1.7pt}
                }
                {
								    \ifthenelse{\equal{\type}{d}}
                    {
                      \draw[densely dashed] (\pos,0) -- (\pos,\height+1);
                    }
										{
											\drawhorizontallines{\width}{\pos}{1.7pt}
										}
                }
            }

            \foreach \pos/\type in {#3}
            {
                \ifthenelse{\equal{\type}{v}}
                {
                    \drawverticallines{\height}{\pos}{0.6pt}
                }
                {
										\drawhorizontallines{\width}{\pos}{0.6pt}
                }
            }

            \drawsolidshadedbox{#8}
            \drawspecialbox{#7}

            \foreach \x/\y/\type in {#5}
            {
                \ifthenelse{\equal{\type}{a}}
                {
                    \draw (\x,\y) circle (6pt);
                    \filldraw (\x,\y) circle (3pt);
                }
                {
                    \filldraw (\x,\y) circle (4pt);
                }
            }
					}
        \end{tikzpicture}
    }
}

\newcommandx{\strule}[3][3={}]
{
    \scalebox{#1}
    {
        \begin{tikzpicture}[baseline=(current bounding box.center)]
            \drawthegridstrule{#2}
            \drawspecialboxcorner{#3}
        \end{tikzpicture}
    }
}

\newcommandx{\dpatt}[9][6={},7={},8={},9={}]
{
    \scalebox{#1}
    {
        \begin{tikzpicture}[baseline=(current bounding box.center)]
					\foreach \width/\height in {#2}
					{
						\useasboundingbox (0.0,-.3) rectangle (\width+1,\height+1.3);
            \shadetheboxesPM{#6}

            \foreach \pos/\type in {#4}
            {
                \ifthenelse{\equal{\type}{v}}
                {
                    \drawverticallines{\height}{\pos}{1.7pt}
                }
                {
								    \ifthenelse{\equal{\type}{d}}
                    {
                      \draw[densely dashed] (\pos,0) -- (\pos,\height+1);
                    }
										{
											\drawhorizontallines{\width}{\pos}{1.7pt}
										}
                }
            }

            \foreach \pos/\type in {#3}
            {
                \ifthenelse{\equal{\type}{v}}
                {
                    \drawverticallines{\height}{\pos}{0.6pt}
                }
                {
										\drawhorizontallines{\width}{\pos}{0.6pt}
                }
            }

            \drawsolidshadedbox{#8}
            \drawspecialbox{#7}

            \foreach \x/\y/\type in {#5}
            {
                \ifthenelse{\equal{\type}{a}}
                {
                    \draw9 (\x,\y) circle (6pt);
                    \filldraw (\x,\y) circle (3pt);
                }
                {
                    \filldraw (\x,\y) circle (4pt);
                }
            }

						\drawlabels{#9}
					}
        \end{tikzpicture}
    }
}

\newcommandx{\shpatt}[7][4={},5={},6={},7={}]
{
	\scalebox{#1}
	{
		\begin{tikzpicture}[baseline=(current bounding box.center)]
			\useasboundingbox (0.0,-.3) rectangle (#2+1,#2+1.3);
			\shadetheboxesPM{#4}
			\draw (0.01,0.01) grid (#2+1-0.01,#2+1-0.01);

			\drawclpattern{#3}{4}

			\foreach \x/\y in {#5}
			{
				\draw[line width=1] (\x,\y) circle (7 pt);
			}

            \foreach \x/\y in {#6}
            {
                \draw[densely dashed, line width=1.7pt] (\x,0) -- (\x,#2+1);
                \draw[densely dashed, line width=1.7pt] (0,\y) -- (#2+1,\y);
            }

            \foreach \xa/\ya/\xb/\yb in {#7}
            {
                \draw[->, line width=1.7pt] (\xa,\ya) -- (\xb-0.12,\yb-0.12);
            }
		\end{tikzpicture}
	}
}

\pgfmathsetmacro{\patttablescale}{1.05}
\pgfmathsetmacro{\pattdispscale}{0.60}
\pgfmathsetmacro{\patttextscale}{0.3}

\pgfmathsetmacro{\struleds}{0.50} 

\theoremstyle{plain}
\newtheorem{theorem}{Theorem}
\newtheorem{proposition}{Proposition}

\theoremstyle{definition}

\newtheorem{question}{Question}

\newcommand{\av}{\mathrm{Av}}
\newcommand{\grid}{\mathrm{Grid}}

\newcommand{\struct}{\textsf{Struct}}
\newcommand{\subclasses}[1]{\mathbb{M}(#1)}
\newcommand{\subpatterns}{\Delta}  

\usepackage{mathtools}

\newcommand{\onebox}{\strule{\struleds}{1/1}[]}

\newcommand{\mcA}{{\mathcal{A}}}

\newcommand{\mcEr}{{\mcE^\mathrm{r}}}
\newcommand{\mcEi}{{\mcE^\mathrm{i}}}
\newcommand{\mcC}{{\mathcal{C}}}
\newcommand{\mcD}{{\mathcal{D}}}
\newcommand{\mcE}{{\mathcal{E}}}
\newcommand{\mcF}{{\mathcal{F}}}
\newcommand{\mcG}{{\mathcal{G}}}
\newcommand{\mcH}{{\mathcal{H}}}
\newcommand{\mcI}{{\mathcal{I}}}

\newcommand{\mcS}{{\mathcal{S}}}
\newcommand{\mcR}{{\mathcal{R}}}
\newcommand{\mcV}{{\mathcal{V}}}

\newcommand{\incr}{{\mathcal{I}}}
\newcommand{\decr}{{\mathcal{D}}}

\title{Automatic discovery of structural rules of permutation classes} %
\author[Christian Bean \and Bjarki Gudmundsson \and Henning Ulfarsson]{Christian Bean \and Bjarki Agust Gudmundsson \and Henning Ulfarsson}

\begin{document}

\begin{abstract}

    We introduce an algorithm that conjectures the structure of a permutation
    class in the form of a disjoint cover of ``rules''; similar to generalized
    grid classes. The cover is usually easily verified by a human and translated
    into an enumeration. The algorithm is successful on different inputs than
    other algorithms and can succeed with any polynomial permutation class. We
    apply it to every non-polynomial permutation class avoiding a set of length
    four patterns. The structures found by the algorithm can sometimes allow an
    enumeration of the permutation class with respect to permutation statistics,
    as well as choosing a permutation uniformly at random from the permutation
    class. We sketch a new algorithm formalizing the human verification of the
    conjectured covers.

\end{abstract}

\maketitle

\section{Introduction}
  \label{sec:introduction}

There has been a movement towards a generalized approach for enumerating
permutation classes, which are sets of permutations defined by the avoidance of
permutation patterns\footnote{We define these precisely at the end of the
introduction.}.

One of the first completely automatic approaches was enumeration schemes
introduced by \cite{originalenumerationschemes} and considerably extended by
\cite{enumerationschemes}. The goal of enumeration schemes is to break up a
permutation class into smaller parts and find recurrence relations. There is
no general theory for when a permutation class has an enumeration scheme.

The insertion encoding is an encoding of finite permutations, introduced by
\cite{insertionencoding}. It encodes how a permutation is built up by
iteratively adding a new maximum element. In particular, they studied the
permutation classes whose insertion encodings are regular languages, including
giving a characterization of these permutation classes. For regular insertion
encodable permutation classes, \cite{practicalinsertionencoding} provides an
algorithm for automatically computing the rational generating function.

\cite{substitutiondecomposition} introduced the enumeration of a set of
permutations by inflating the simple permutations in the set. A basis is not
required but the set under inspection must contain only finitely many simple
permutations. Checking whether this condition is satisfied is in general
exponentially hard, but for wreath-closed permutation classes an $O(n \log n)$
algorithm exists, \cite{testsimple}. The whole procedure has not been
implemented.

In the case of polynomial permutation classes\footnote{A permutation class
$\mcC$ is said to be \emph{polynomial} if the number of length $n$ permutations,
$|\mcC_n|$, is given by a polynomial for all sufficiently large $n$.}, it was
shown by \cite{polynomialenumeration}, by combining results by
\cite{geometricgridclasses} and \cite{fibonaccistructure}, that each permutation
class can be represented by a finite set of peg permutations\footnote{We discuss
peg permutations in Section~\ref{sec:polynomial_bases}}. From this finite set of
peg permutations, \cite{polynomialenumeration} give an automatic method to
enumerate polynomial permutation classes. However, it is not known in general how
to find the set of peg permutations.

We introduce \struct{}, an algorithm which takes as input a set of permutations
or a permutation property and conjectures certain structural rules that can
lead to a generating function enumerating the input. We focus specifically on
permutation classes and assume a finite basis is known. We emphasize if the
algorithms mention above succeed, then they output a proof, whereas our
algorithm only outputs conjectures. We note however that the produced
conjectures usually turn out to be easily verified by a human.

As a first example, consider the set $\mcA$ of permutations avoiding $231$,
which are the stack-sortable permutations, first considered by \cite{knuth}. It
is well known that a non-empty permutation in $\mcA$ can be written as $\alpha n
\beta$ where $n$ is the largest element in the permutation, $\alpha$ and $\beta$
avoid $231$, and every element in $\alpha$ is smaller than every element in
$\beta$. We can state the structure as a disjoint union of the two rules in
Figure~\ref{fig:231}.
    \begin{figure}[ht]
        \centering
        \[
            \mcA =  \onebox \sqcup \strule{\struleds}{3/3}[0/0/$\mcA$,1/2/$\bullet$,2/1/$\mcA$]. 
        \]
        \caption{The structure of $\av(231)$.}
        \label{fig:231}
    \end{figure}
The meaning of this representation is that a permutation is either empty, and
represented by the empty rule, or, it is non-empty and represented by the
$3 \times 3$ rule. That rule is a recipe that says: given two permutations in
$\mcA$, we can generate a longer permutation in $\mcA$ by placing the two
permutations in the cells marked with $\mcA$ and placing a new largest element
between them, e.g.~$213$ and $1234$ will generate the permutation $21384567$.
This leads immediately to the well-known functional equation for the Catalan
generating function, $A(x) = 1 + x A(x)^2$.

For the remainder of the introduction, we give the definitions used throughout.
In Section~\ref{sec:the_algorithm} we discuss precisely how the algorithm works.
In Sections~\ref{sec:patterns_of_length_three} and
\ref{sec:one_pattern_of_length_three_and_one_of_length_four} we apply \struct{}
to specific permutation classes. In Section~\ref{sec:polynomial_bases} we
compare \struct{} to the other existing approaches mentioned above and show it
will succeed for all polynomial permutation classes. In
Section~\ref{sec:large_bases_of_length_four_patterns} we apply it to the bases
$B \subseteq \mcS_4$. In Section~\ref{sec:further_improvements} we consider
improvements of the algorithm, relating to permutation statistics, uniform
random choice from a permutation class, and how the conjectures can be turned
automatically into theorems.

A \emph{permutation} of length $n$ is a word on $[n] = \{1,2, \dots, n\}$ that
contains each letter exactly once. Let $\mcS_n$ be the set of all permutations
of length $n$ and $\mcS = \bigcup_{n \geq 0} \mcS_n$. Two words of the same
length, $w = w_1 w_2 \dots w_k$ and $v = v_1 v_2 \dots v_k$, are \emph{order
isomorphic} if for all $i$ and $j$, $w_i < w_j$ if and only if $v_i < v_j$. A
permutation $p$ is \emph{contained} in a permutation $\pi$, denoted $p \preceq
\pi$, if there is a subword in $\pi$ that is order isomorphic to $p$. If $\pi$
does not contain $p$ then we say $\pi$ \emph{avoids} $p$. In this context, $p$ is
referred to as a (\emph{classical permutation}) \emph{pattern}.

Define the set $\subpatterns(\pi)$ to be the set of all patterns contained in
$\pi$,
\[
  \subpatterns(\pi) = \{ p \in \mcS : p \preceq \pi \}.
\]
For a set of patterns $P$, let
$\subpatterns(P) = \bigcup_{\pi \in P} \subpatterns(\pi)$.
A \emph{permutation class} is a set $\mcC$ such that $\subpatterns(\mcC)$ equals
$\mcC$. For a set of patterns $P$, the set
\[
  \av(P) = \{ \pi \in \mcS : \pi \text{ avoids all } p \in P \}
\]
is a permutation class. Given a permutation class $\mcC = \av(B)$ we say that a
set of patterns $B$ is a \emph{basis} of $\mcC$ if $\subpatterns(p) \cap B = \{ p \}$
for each $p \in B$. It is easy to see that every permutation class has a
(possibly infinite) basis. If $|B| = 1$ then we call $\av(B)$ a \emph{principal}
permutation class.

For a permutation class $\mcC = \av(P)$ we denote the set of all length $n$
permutations in $\mcC$ as $\mcC_n$ or $\av_n(P)$. We also let $\mcC_{\leq n}$
and $\av_{\leq n}(P)$ be the set of permutations of length at most $n$ in the
permutation class. Two sets of patterns $P$ and $P'$ are \emph{Wilf-equivalent}
if $|\av_n(P)| = |\av_n(P')|$ for all $n$.

\section{The algorithm}
  \label{sec:the_algorithm}

Before describing how the algorithm works we give a few more definitions. A
\emph{block} of a permutation class $\mcC = \av(B)$ is a permutation class
$\mcC' = \av(B')$, containing infinitely many permutations, such that $B'
\subseteq \subpatterns(B)$ and $B' \cap \subpatterns(p) \neq \emptyset$ for all
$p$ in $B$. We also allow the finite permutation class $\av(1) = \{\epsilon\}$
as a block for any permutation class $\mcC$. Additionally, if $1 \in \mcC$ we
allow $\{1\}$ as a block, even though this is not a permutation class, and
contains only one permutation. We call $1$ ``the point'' and denote it with
$\bullet$.

The \emph{block set}, $\subclasses{B}$, of $\mcC$ is the set of all blocks of
$\mcC$. We note that for a finite basis $B$ the block set of the permutation
class $\av(B)$ is always finite. For example,
\[
\subclasses{ \{231\} } = \{
\av(1), \{ \bullet \}, \av(12), \av(21), \av(231) \}.
\]
Note that $\av(12, 21)$ is not a block since it is finite.\\

Before we assemble blocks into ``rules'' we review generalized grid classes
introduced by \cite{generalizedgridclasses}. Given a permutation $\pi$ of length
$n$, and two subsets $X,Y \subseteq [n]$, then $\pi(X \times Y)$ is the
permutation that is order isomorphic to the subword with indices from $X$ and
values in $Y$. For example $35216748([3,7] \times [2,6]) = 132$, from the
subword $264$.

Suppose $M$ is a $t \times u$ matrix (indexed from left to right and bottom to
top ) whose entries are permutation classes. An \emph{$M$-gridding} of a
permutation $\pi$ of length $n$ is a pair of sequences $1 = c_1 \leq \dots \leq
c_{t+1} = n+1$ and $1 = r_1 \leq \dots \leq r_{u+1} = n+1$ such that
$\pi([c_k,c_{k+1}) \times [r_l,r_{l+1}))$ is in $M_{k,l}$ for all $k$ in $[t]$
and $\ell$ in $[u]$. The \emph{generalized grid class} of $M$, $\grid(M)$,
consists of all permutations with an $M$-gridding.

A (\emph{\struct{}}) \emph{rule} $\mcR$ is a matrix whose entries are
permutation classes, or the point $\bullet$, with the requirement that each
permutation $\pi$ in the grid class $\grid(\mcR)$ has a unique $\mcR$-gridding.
From now on we will abuse notation and use $\mcR$ to denote both the \struct{}
rule and its grid class. Let $\mcR_{\leq n}$ be the set of permutations in
$\mcR$ of length at most $n$.

For a given basis $B$, whose longest pattern is from $S_\ell$, \struct{} tries
to find the structure of the permutation class $\mcC = \av(B)$ in terms of
rules. Using the default settings, the algorithm consists of four main steps
\begin{enumerate}
  \item Generate the block set\footnote{Candidates for blocks are experimentally checked
  for being infinite: If a block is non-empty up to length $\ell+2$ it is kept.} of the given basis $B$.

  \item Generate \struct{} rules\footnote{Candidates for \struct{} rules are
  experimentally checked to not create duplicates up to length $\ell+2$.},
  $\mcR_1, \mcR_2, \dots, \mcR_K$ up to size $(\ell+1) \times (\ell+1)$ with
  entries from the block set, satisfying $\mcR_{\leq \ell+2} \subseteq \av(B)$.
  The trivial rule $\strule{\struleds}{1/1}[0/0/$\mcC$]$ is discarded.

  \item Try to find a \emph{cover} of $\av_{\leq \ell+2}(B)$ with the rules from
  the previous step, i.e., write this set as a disjoint union
  \[
    \av_{\leq \ell+2}(B) = \mcR_{i_1,\leq \ell+2} \sqcup \dots \sqcup \mcR_{i_k,\leq \ell+2}.
  \]
  \item If a cover is found, verify that it remains valid in length $ \leq \ell+4$, i.e.,
  \[
    \av_{\leq \ell+4}(B) = \mcR_{i_1,\leq \ell+4} \sqcup \dots \sqcup \mcR_{i_k,\leq \ell+4}.
  \]
\end{enumerate}
By providing different settings, the user can make the algorithm look for a
cover with larger rules, or verify a found cover for longer permutations.

As a simple example consider the decreasing permutations $\decr = \av(12)$,
where $\ell = 2$. Step $(1)$ finds the block set
$\subclasses{ \{12\} } = \{ \av(1), \{ \bullet \}, \av(12) \}$.
Step $(2)$ generates the rules
    \[
        \onebox, \strule{\struleds}{2/2}[0/1/$\bullet$,1/0/$\bullet$],
                 \strule{\struleds}{2/2}[0/1/$\decr$,1/0/$\bullet$],
                 \strule{\struleds}{2/2}[0/1/$\bullet$,1/0/$\decr$], \dotsc,
                 \strule{\struleds}{3/3}[0/2/$\bullet$,1/1/$\bullet$,2/0/$\decr$].
    \]
One of the covers found in step $(3)$ is
    \[
      \decr_{\leq 4} = \onebox \sqcup
      \strule{\struleds}{2/2}[0/1/$\decr$,1/0/$\bullet$]_{\leq 4}.
    \]
Step $(4)$ verifies this cover up to length $6$. At this stage, a human must
step in and verify that the cover remains valid for all lengths. In this case,
it is obvious, and if $D(x)$ is the generating function then we can see from the
cover that $D(x) = 1 + xD(x)$, so $D(x) = \frac{1}{1-x}$.

If the block set is large then the space of possible rules to consider in step
(2) will be too large to exhaustively search for valid rules. To prune the
search space we first check which blocks can share a row, column or diagonal
without creating a pattern from the basis. This is used to recursively build the
candidate \struct{} rules without considering every possibility. We also arrange
the blocks in a poset where the relation is set containment. This is also used
when creating the \struct{} rules; if a rule with a block $A$ in a certain cell
produces the same permutation twice, or a permutation outside of $\av(B)$ then
replacing $A$ with a block $A' \supset A$ will also not work. Finally, in step
(3) we use the integer linear programming solver Gurobi (\cite{gurobi}) to find
a minimal cover if one exists; resorting to the SAT solver Lingeling
(\cite{lingeling}) if Gurobi runs out of memory.

The implementation of \struct{} can be found on GitHub, \cite{structrepo}. The
conjectures that have been discovered can be found on the Permutation Pattern
Avoidance Library (or PermPAL for short), \cite{permpal}, alongside the
conjectured enumerations.

\section{Patterns of length three}
  \label{sec:patterns_of_length_three}

As we saw in the previous section the decreasing permutations, $\mcD = \av(12)$,
have a cover from which we can read the generating function, $D(x) =
\frac{1}{1-x}$. By symmetry the increasing permutations $\mcI = \av(21)$ have
the same generating function, which we denote $I(x)$. Throughout we will use the
font $\mcC$ to denote a permutation class and $C(x)$ to denote its corresponding
generating function.

In this section, we consider selected bases $B \subseteq \mcS_3$. All of the
results in this section are known, see e.g., \cite{restrictedpermutations}. We
include them here as we will need to reference them for permutation classes
considered later. For the bases not mentioned in this section there exists a
cover using at most $3 \times 3$ rules, with the exception of $\{ 123 \}$, which
will be treated in Section~\ref{sec:further_improvements}. Our approach has the
appeal of being automatic: Every cover is conjectured by \struct{} and easy for
a human to prove. Furthermore, reading the generating function from the cover is
a routine exercise. Where possible the OEIS entry is referenced, \cite{OEIS}.

We already have from Figure~\ref{fig:231} that $\mcA = \av(231)$ has a
cover that leads to the Catalan generating function, that is $A(x) = \frac{1 -
\sqrt{1-4x}}{2x}$.

\subsection*{Two patterns}
We start by considering two patterns of length three.

\begin{proposition}[\cite{restrictedpermutations}, Proposition 11]
    The structure of the permutation class $\mcE_2 = \av(123, 231)$ is given by
    the cover in Equation~\eqref{eq:123_231}. The enumeration is given by
    \oeis{A152947} and the generating function is
    \begin{align*}
        E_2(x) = \frac{1 - 2x + 2x^2}{(1-x)^3}
               = 1 &+ x + 2x^2 + 4x^3 + 7x^4 + 11x^5 + 16x^6\\
                   &+ 22x^7 + 29x^8 + 37x^9 + 46x^{10} + \dotsb.
    \end{align*}
\end{proposition}
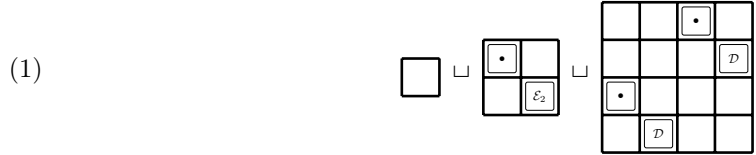
\begin{equation} \label{eq:123_231}
    \onebox \sqcup
    \strule{\struleds}{2/2}[0/1/$\bullet$, 1/0/$\mcE_2$] \sqcup
    \strule{\struleds}{4/4}[0/1/$\bullet$, 2/3/$\bullet$, 1/0/$\decr$, 3/2/$\decr$]
\end{equation}

\begin{proof}
    The structure is easily obtained by starting with the structure of
    $231$-avoiders in Figure~\ref{fig:231}. From this structure we get the
    functional equation $E_2(x) = 1 + xE_2(x) + x^2D(x)^2$. Solving it gives the
    claimed equation for $E_2(x)$.
\end{proof}

We then consider three Wilf-equivalent bases.

\begin{proposition} [\cite{restrictedpermutations}, Propositions 7, 8 and 9] \label{prop:123_132_and_Wilf_eq}
    The structures of the permutation classes $\mcE_{3,1}$, $\mcE_{3,2}$ and
    $\mcE_{3,3}$ are given in Table~\ref{tab:B31_B33}. The enumeration is given
    by \oeis{A011782} and the generating function is
    \begin{align*}
        E_{3,1}(x) = \frac{1-x}{1-2x}
        = 1 &+ x + 2x^2 + 4x^3 + 8x^4 + 16x^5 + 32x^6\\
            &+ 64x^7 + 128x^8 + 256x^9 + 512x^{10} + \dotsb.
    \end{align*}
\end{proposition}

\begin{table}[ht]
  \centering
  \begin{tabular}{|l|c|}
      \hline
      $\mcE_{3,1} = \av(123, 132)$ &
      $
      \onebox \sqcup
      \strule{\struleds}{3/3}[0/1/$\decr$, 1/2/$\bullet$, 2/0/$\mcE_{3,1}$]
      $\\
      \hline
      $\mcE_{3,2} = \av(132, 312)$ &
      $
      \onebox \sqcup
      \strule{\struleds}{3/3}[0/1/$\mcE_{3,2}$, 1/2/$\bullet$, 2/0/$\decr$] =
      \onebox \sqcup
      \strule{\struleds}{2/3}[0/1/$\bullet$, 1/0/$\decr$, 1/2/$\incr$]
      $\\
      \hline
      $\mcE_{3,3} = \av(231, 312)$ &
      $
      \onebox \sqcup
      \strule{\struleds}{3/3}[0/0/$\mcE_{3,3}$, 1/2/$\bullet$, 2/1/$\decr$]
      $\\
      \hline
  \end{tabular}
  \caption{The structure of $\av(B)$ where $B$ ranges over the non-symmetric
  bases in a Wilf-class containing two length three patterns}
  \label{tab:B31_B33}
\end{table}

\begin{proof}
  From the cover of $\mcE_{3,1}$ in Table~\ref{tab:B31_B33} we obtain
  $E_{3,1}(x) = 1 + x D(x) E_{3,1}(x)$. Solving produces the claimed equation.
  For $\mcE_{3,2}$ we give two covers, the left-hand cover
  being similar to as above. For the right-hand cover, we have $\mcI$ and $\mcD$
  mixing. This mixing is counted by the generating function $\frac{1}{1-2x}$ and
  we obtain $E_{3,2}(x) = 1 + \frac{x}{1-2x} = E_{3,1}(x)$. The cover of
  $\mcE_{3,3}$ is similar to the one of $\mcE_{3,1}$.
\end{proof}

\subsection*{Three patterns}
The final basis we consider is counted by the Fibonacci numbers.

\begin{proposition}[\cite{restrictedpermutations}, Proposition 15]
    The structure of the permutation class $\mcF = \av(123, 132, 213)$ is
    given by the cover in Equation~\eqref{eq:123_132_213}. They are enumerated
    by the Fibonacci numbers (\oeis{A000045}, shifted) and the generating
    function is
    \begin{align*}
        F(x) = \frac{1}{1-x-x^2}
               =  1 &+ x + 2x^2 + 3x^3 + 5x^4 + 8x^5 + 13x^6\\
                    &+ 21x^7 + 34x^8 + 55x^9 + 89x^{10} + \dotsb.
    \end{align*}
\end{proposition}
\begin{equation} \label{eq:123_132_213}
    \onebox \sqcup
    \strule{\struleds}{2/2}[0/1/$\bullet$, 1/0/$X$] \sqcup
    \strule{\struleds}{3/3}[0/1/$\bullet$, 1/2/$\bullet$, 2/0/$X$]
\end{equation}

\begin{proof}
    From the cover, we obtain $F(x) = 1 + x F(x) + x^2 F(x)$. Solving gives
    the claimed equation.
\end{proof}

\section{One pattern of length three and one of length four}
  \label{sec:one_pattern_of_length_three_and_one_of_length_four}

In this section, we consider bases consisting of one length three pattern and
one length four pattern. These permutation classes were enumerated by
\cite{generatingtrees} and \cite{restrictedatkinson}. We consider them here with
an emphasis on an automated and uniform approach, and because we need to refer
to some of them later.

There are eighteen non-symmetric cases, of which \struct{} can find a cover for
sixteen\footnote{The remaining two are treated in
Section~\ref{sec:further_improvements}}. The cover is easily verified, and a
functional equation can be written for the generating function. As the methods
are similar in all the proofs we only give the details in one case, for
$\av(132,4231)$ in Proposition~\ref{prop:132_4312+132_4231}.

\subsubsection*{Wilf-class $1$}
There is a single basis in this Wilf-class: $\{321,1234\}$. There are
finitely many avoiding permutations. \struct{} finds $21$ rules, the largest of
which is $6$-by-$6$.

\subsubsection*{Wilf-class $2$}
There is a single basis in this Wilf-class.
\begin{proposition}[\cite{restrictedatkinson}, Proposition 3.1]
    \label{prop:321_2134}
    The structure of the permutation class $\mcG_2 = \av(321, 2134)$ is given by the cover in
    Equation~\eqref{eq:321_2134}. The enumeration is given by \oeis{A116699} and
    the generating function is
    \begin{align*}
        G_2(x) = \frac{1 - 4x + 7x^2 - 5x^3 + 3x^4 - x^5}{1 - 5x + 10x^2 - 10x^3 + 5x^4 - x^5}
               =  1 &+ x + 2x^2 + 5x^3 + 13x^4 + 30x^5 + 61x^6\\
                    &+ 112x^7 + 190x^8 + 303x^9 + 460x^{10} + \dotsb.
    \end{align*}
\end{proposition}
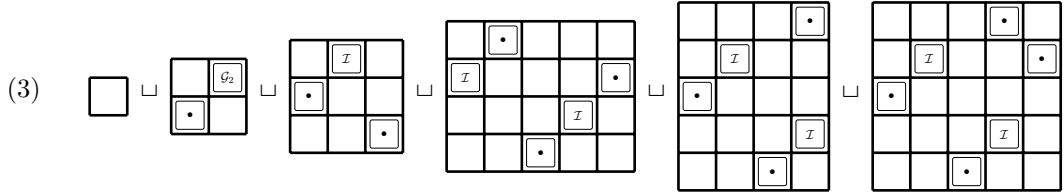
\begin{equation} \label{eq:321_2134}
    \onebox \sqcup
    \strule{\struleds}{2/2}[0/0/$\bullet$, 1/1/$\mcG_2$] \sqcup
    \strule{\struleds}{3/3}[0/1/$\bullet$, 1/2/$\incr$,   2/0/$\bullet$] \sqcup
    \strule{\struleds}{5/4}[0/2/$\incr$,   1/3/$\bullet$, 2/0/$\bullet$, 3/1/$\incr$, 4/2/$\bullet$] \sqcup
    \strule{\struleds}{4/5}[0/2/$\bullet$, 1/3/$\incr$,   2/0/$\bullet$, 3/1/$\incr$, 3/4/$\bullet$] \sqcup
    \strule{\struleds}{5/5}[0/2/$\bullet$, 1/3/$\incr$,   2/0/$\bullet$, 3/1/$\incr$, 3/4/$\bullet$, 4/3/$\bullet$]
\end{equation}

\begin{proof}
    From the cover, we obtain $G_2(x) = 1 + xG_2(x) + x^2I(x) + 2x^3I(x)^3 + x^4I(x)^4$.
    Solving gives the claimed equation.
\end{proof}

\subsubsection*{Wilf-class $3$}
There is a single basis in this Wilf-class.
\begin{proposition}[\cite{restrictedatkinson}, Proposition 3.3]
    The structure of the permutation class $\mcG_3 = \av(132, 4321)$ is given by the cover in
    Equation~\eqref{eq:132_4321}. The enumeration is given by \oeis{A116701} and
    the generating function is
    \begin{align*}
        G_3(x) = \frac{1 - 4x + 7x^2 - 5x^3 + 3x^4}{1 - 5x + 10x^2 - 10x^3 + 5x^4 - x^5}
               =  1 &+ x + 2x^2 + 5x^3 + 13x^4 + 31x^5 + 66x^6\\
                    &+ 127x^7 + 225x^8 + 373x^9 + 586x^{10} + \dotsb.
    \end{align*}
\end{proposition}
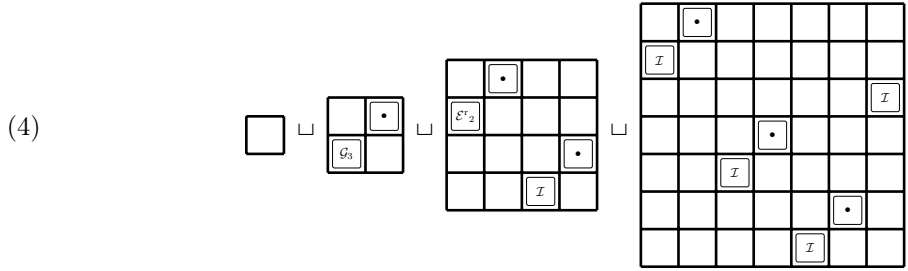
\begin{equation} \label{eq:132_4321}
    \onebox \sqcup
    \strule{\struleds}{2/2}[0/0/$\mcG_3$,  1/1/$\bullet$] \sqcup
    \strule{\struleds}{4/4}[0/2/$\mcEr_2$, 1/3/$\bullet$, 2/0/$\incr$, 3/1/$\bullet$] \sqcup
    \strule{\struleds}{7/7}[0/5/$\incr$,   1/6/$\bullet$, 2/2/$\incr$, 3/3/$\bullet$, 4/0/$\incr$, 5/1/$\bullet$, 6/4/$\incr$]
\end{equation}
Here $\mcEr_2 = \av(132, 321)$ is the reverse of $\mcE_2$, and so they have the
same enumeration.

\begin{proof}
    From the cover, we obtain $G_3(x) = 1 + xG_3(x) + x^2E_2(x)I(x) + x^3I(x)^4$.
    Solving produces the claimed equation.
\end{proof}

\subsubsection*{Wilf-class $4$}
There is a single basis in this Wilf-class.

\begin{proposition}[\cite{restrictedatkinson}, Proposition 3.2]
    The structure of the permutation class $\mcG_4 = \av(321, 1324)$ is given by the cover in
    Equation~\eqref{eq:321_1324}. The enumeration is given by \oeis{A179257} and
    the generating function is
    \begin{align*}
        G_4(x) = \frac{1 - 5x + 11x^2 - 12x^3 + 8x^4 - 2x^5}{1 - 6x + 15x^2 - 20x^3 + 15x^4 - 6x^5 + x^6}
               =  1 &+ x + 2x^2 + 5x^3 + 13x^4 + 32x^5 + 72x^6\\
                    &+ 148x^7 + 281x^8 + 499x^9 + 838x^{10} + \dotsb.
    \end{align*}
\end{proposition}

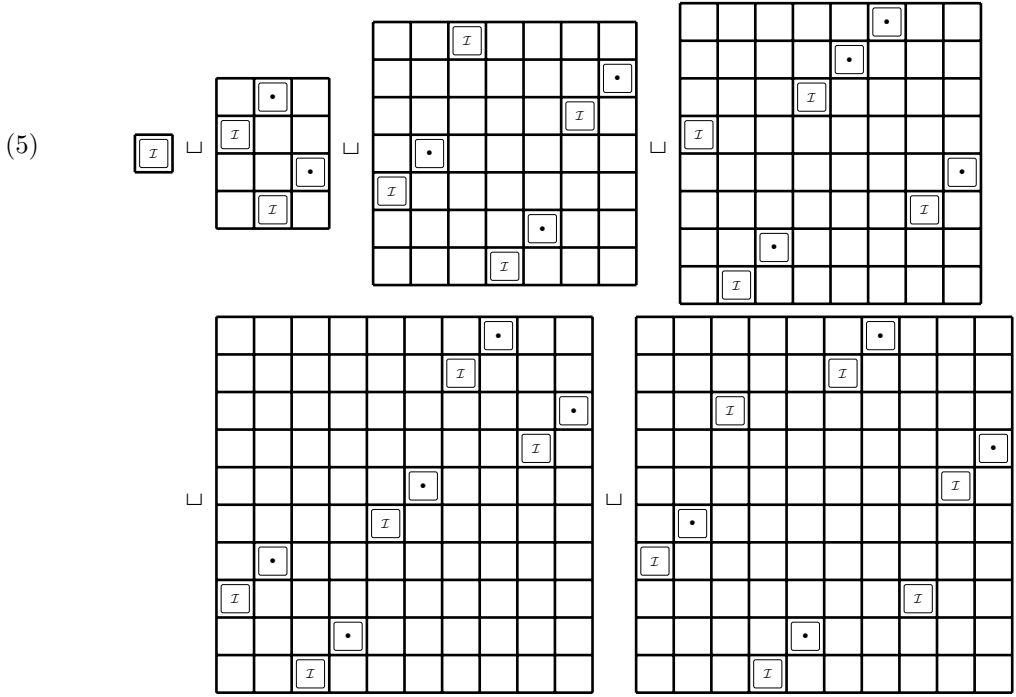
\begin{align}
\label{eq:321_1324}
    \strule{\struleds}{1/1}[0/0/$\incr$] &\sqcup
    \strule{\struleds}{3/4}[0/2/$\incr$, 1/0/$\incr$, 1/3/$\bullet$, 2/1/$\bullet$] \sqcup
    \strule{\struleds}{7/7}[0/2/$\incr$, 1/3/$\bullet$, 2/6/$\incr$, 3/0/$\incr$, 4/1/$\bullet$, 5/4/$\incr$, 6/5/$\bullet$] \sqcup
    \strule{\struleds}{8/8}[0/4/$\incr$, 1/0/$\incr$, 2/1/$\bullet$, 3/5/$\incr$, 4/6/$\bullet$, 5/7/$\bullet$, 6/2/$\incr$, 7/3/$\bullet$] \\ \notag
    &\sqcup \strule{\struleds}{10/10}[0/2/$\incr$, 1/3/$\bullet$, 2/0/$\incr$, 3/1/$\bullet$, 4/4/$\incr$, 5/5/$\bullet$, 6/8/$\incr$, 7/9/$\bullet$, 8/6/$\incr$, 9/7/$\bullet$] \sqcup
    \strule{\struleds}{10/10}[0/3/$\incr$, 1/4/$\bullet$, 2/7/$\incr$, 3/0/$\incr$, 4/1/$\bullet$, 5/8/$\incr$, 6/9/$\bullet$, 7/2/$\incr$, 8/5/$\incr$, 9/6/$\bullet$]
\end{align}

\begin{proof}
    From the cover, we obtain
    $G_4(x) =  I(x) + x^2 I(x)^3 + x^3 I(x)^4 + x^4 I(x)^4 + x^5 I(x)^5 + x^4 I(x)^6$.
    Solving produces the claimed equation.
\end{proof}

\subsubsection*{Wilf-class $5$}
There is a single basis in this Wilf-class.

\begin{proposition}[\cite{generatingtrees}]
    The structure of the permutation class $\av(321, 1342)$ is given by the cover in
    Equation~\eqref{eq:321_1342}. The enumeration is given by \oeis{A116702}
    and the generating function is
    \begin{align*}
        G_5(x) = \frac{1 - 4x + 6x^2 - 3x^3 + x^4}{1 - 5x + 9x^2 - 7x^3 + 2x^4}
               =  1 &+ x + 2x^2 + 5x^3 + 13x^4 + 32x^5 + 74x^6\\
                    &+ 163x^7 + 347x^8 + 722x^9 + 1480x^{10} + \dotsb.
    \end{align*}
\end{proposition}

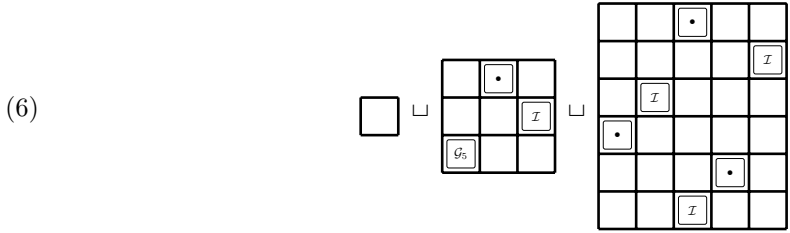
\begin{equation} \label{eq:321_1342}
    \onebox \sqcup
    \strule{\struleds}{3/3}[0/0/$\mcG_5$,  1/2/$\bullet$, 2/1/$\incr$] \sqcup
    \strule{\struleds}{5/6}[0/2/$\bullet$, 1/3/$\incr$,   2/0/$\incr$, 2/5/$\bullet$, 3/1/$\bullet$, 4/4/$\incr$]
\end{equation}

\begin{proof}
    From the cover, we obtain
    $
        G_5(x) = 1 + xG_5(x)I(x) + x^3I(x)^4
    $,
    and solving produces the claimed equation.
\end{proof}

\subsubsection*{Wilf-class $6$} \label{sec:321_2143}

There is a single basis in this Wilf-class.

\begin{proposition}[\cite{generatingtrees}]
  \label{prop:321_2143}
  The structure of the permutation class $\mcG_6 = \av(321, 2143)$ is given by the cover in Equation~\eqref{eq:321_2143}.
  The enumeration is given by \oeis{A088921}
  and the generating function is
  \begin{align*}
      G_5(x) = \frac{1 - 4x + 5x^2 - x^3}{1 - 5x + 8x^2 - 4x^3}
             = 1 &+ x + 2x^2 + 5x^3 + 13x^4 + 33x^5 + 80x^6 \\
              &+ 185x^7 + 411x^8 + 885x^9 + 1862x^{10} + \dotsb.
  \end{align*}
\end{proposition}

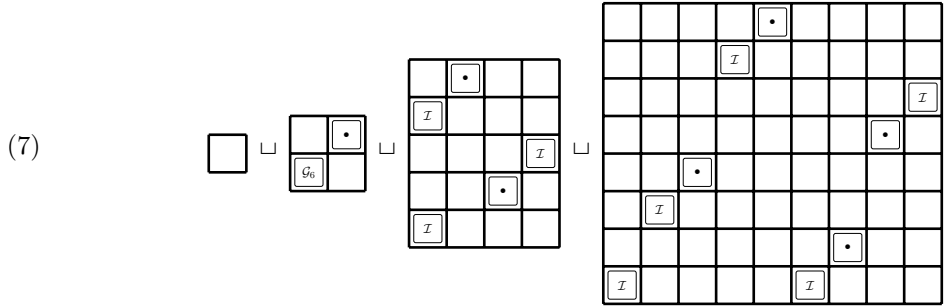
\begin{equation} \label{eq:321_2143}
    \onebox \sqcup
    \strule{\struleds}{2/2}[0/0/$\mcG_6$,  1/1/$\bullet$] \sqcup
    \strule{\struleds}{4/5}[0/0/$\incr$, 0/3/$\incr$, 1/4/$\bullet$, 2/1/$\bullet$, 3/2/$\incr$]\sqcup
    \strule{\struleds}{9/8}[0/0/$\incr$, 1/2/$\incr$, 2/3/$\bullet$, 3/6/$\incr$, 4/7/$\bullet$, 5/0/$\incr$, 6/1/$\bullet$, 7/4/$\bullet$, 8/5/$\incr$]
\end{equation}

\begin{proof}
  From the cover, we obtain
  $
    G_6(x) = 1 + xG_6(x) + \frac{x^2 I(x)}{1-2x} + \frac{x^4I(x)^3}{1-2x}
  $,
  and solving gives the claimed equation.
\end{proof}

\subsubsection*{Wilf-class $7$}

There are two bases in this Wilf-class enumerated by \oeis{A005183} (shifted).
\begin{proposition} [\cite{restrictedatkinson}, Propositions 3.7 and 3.8 of supplement]
    \label{prop:132_4312+132_4231}
    The structures of the permutation classes $\mcG_{7,1} = \av(132, 4312)$ and $\mcG_{7,2} =
    \av(132, 4231)$ are given by the covers in
    Table~\ref{tab:132_4312+132_4231}. Their generating function is
    \begin{align*}
        G_{7,1}(x) = \frac{1 - 4x + 5x^2 - x^3}{1 - 5x + 8x^2 - 4x^3}
                   =  1 &+ x + 2x^2 + 5x^3 + 13x^4 + 33x^5 + 81x^6\\
                        &+ 193x^7 + 449x^8 + 1025x^9 + 2305x^{10} + \dotsb.
    \end{align*}
\end{proposition}

\begin{table}[ht]
  \centering
  \begin{tabular}{|l|c|}
      \hline
       $\mcG_{7,1} = \av(132, 4312)$ &
      $
      \strule{\struleds}{1/1}[0/0/$\incr$] \sqcup
      \strule{\struleds}{5/6}[0/3/$\bullet$, 1/4/$\incr$, 2/1/$\mcE_{3,2}$, 3/2/$\bullet$, 4/0/$\decr$, 4/5/$\incr$]
      $\\
      \hline
      $\mcG_{7,2} = \av(132, 4231)$ &
      $
      \onebox \sqcup
      \strule{\struleds}{2/2}[0/0/$G_{7,2}$, 1/1/$\bullet$] \sqcup
      \strule{\struleds}{5/5}[0/3/$\mcE_{3,2}$, 1/4/$\bullet$, 2/1/$\bullet$, 3/0/$\mcEi_{3,2}$, 4/2/$\incr$]
      $\\
      \hline
  \end{tabular}
  \caption{The structure of $\mcG_{7,1}$ and $\mcG_{7,2}$. Here $\mcEi_{3,2} = \av(132, 231)$ is the inverse symmetry of $\mcE_{3,2}$.}
  \label{tab:132_4312+132_4231}
\end{table}

We will provide a proof for the cover obtained for $\mcG_{7,2}$. It illustrates
the method needed for the other covers found in this section.

\begin{proof}
    From the cover for $\mcG_{7,1}$ in Table~\ref{tab:132_4312+132_4231} we obtain
    \[
        G_{7,1}(x) = I(x) + \frac{x^2 I(x) E_{3,2}(x)}{1-2x}.
    \]

    Solving produces the claimed equation. We will now provide proof for the
    cover of $\mcG_{7,2}$. A permutation of length $n \geq 1$ in $\mcG_{7,2}$ is
    of the form $\alpha n \beta$. If $\beta$ is the empty permutation then
    $\alpha$ can be any permutation in $\mcG_{7,2}$. Otherwise, the permutation
    is of the form $\alpha n k \beta \gamma$ where all the letters in $\alpha$
    are greater than those in $\beta$ and $\gamma$, else you would have an
    occurrence of $132$ using $n$, and all the elements in $\beta$ are less than
    those in $\gamma$, else you would have an occurrence of $4231$ using $n$ and
    $k$. In order to avoid $132$ and $4231$, we have $\alpha \in \mcE_{3,2}$,
    $\beta \in \av(132, 231)$, and $\gamma \in \av(21)$.  This is the cover for
    $\mcG_{7,2}$ in Table~\ref{tab:132_4312+132_4231}. From this cover we obtain
    \[
        G_{7,2}(x) = 1 + x G_{7,2}(x) + x^2 I(x) E_{3,2}(x)^2,
    \]
    and solving gives the claimed equation.
\end{proof}

\subsubsection*{Wilf-class $8$}
There is a single basis in this Wilf-class.

\begin{proposition}[\cite{restrictedatkinson}, Proposition 3.9 of supplement]
    The structure of the permutation class $\av(132, 3214)$ is given by the cover in
    Equation~\eqref{eq:132_3214}. The enumeration is given by \oeis{A116703} and
    the generating function is
    \begin{align*}
        G_8(x) = \frac{1 - 3x + 3x^2 - x^3}{1 - 4x + 5x^2 - 3x^3}
               =  1 &+ x + 2x^2 + 5x^3 + 13x^4 + 33x^5 + 82x^6\\
                    &+ 202x^7 + 497x^8 + 1224x^9 + 3017x^{10} + \dotsb.
    \end{align*}
\end{proposition}

\begin{equation} \label{eq:132_3214}
    \onebox \sqcup
    \strule{\struleds}{3/3}[0/1/$\mcEr_2$, 1/2/$\bullet$, 2/0/$\mcG_8$]
\end{equation}

\begin{proof}
    From the cover, we obtain
    $
        G_8(x) = 1 + x E_2(x) G_8(x)
    $,
    and solving gives the claimed equation.
\end{proof}

\subsubsection*{Wilf-class $9$}
There are 9 bases in this Wilf-class, enumerated by a bisection of the Fibonacci
numbers \oeis{A001519}.

\begin{proposition} [\cite{restrictedatkinson}, Propositions 16 and 18 of supplement]
   \label{321_3142_and_132_3412}
    The structure of the permutation classes $\mcG_{9,1} = \av(321, 3142)$, $\mcG_{9,2} = \av(132, 3412)$ are given by
    the covers in Table~\ref{tab:91_and_92}. The generating function is
    \begin{align*}
        G_{9,1}(x) = \frac{1 - 2x}{1 - 3x + x^2}
                   =  1 &+ x + 2x^2 + 5x^3 + 13x^4 + 34x^5 + 89x^6\\
                        &+ 233x^7 + 610x^8 + 1597x^9 + 4181x^{10} + \dotsb.
    \end{align*}
\end{proposition}

\begin{proof}
    From the covers, we obtain
    \[
        G_{9,1}(x) = 1 + x G_{9,1}(x) + \frac{x^2 G_{9,1}}{1 - 2x}.
    \]
    Solving gives the claimed equation.
\end{proof}

\begin{table}[ht]
  \centering
  \begin{tabular}{|l|c|}
      \hline
      $\av(321, 3142)$ &
      $
      \onebox \sqcup
      \strule{\struleds}{2/2}[0/0/$\mcG_{9,1}$, 1/1/$\bullet$] \sqcup
      \strule{\struleds}{5/4}[0/0/$\mcG_{9,1}$, 1/2/$\incr$, 2/3/$\bullet$, 3/1/$\bullet$, 4/2/$\incr$]
      $\\
      \hline
      $\av(132, 3412)$ &
      $
      \onebox \sqcup
      \strule{\struleds}{2/2}[0/1/$\bullet$, 1/0/$\mcG_{9,2}$] \sqcup
      \strule{\struleds}{4/5}[0/2/$\bullet$, 1/1/$\mcG_{9,2}$, 2/3/$\bullet$, 3/4/$\incr$, 3/0/$\decr$]
      $\\
      \hline
  \end{tabular}
  \caption{The structure of $\mcG_{9,1}$ and $\mcG_{9,2}$}
  \label{tab:91_and_92}
\end{table}

\begin{proposition}[\cite{restrictedatkinson}, Propostions 11, 12, 13, 14 and 17 of supplement]
    The structures of the permutation classes $\mcG_{9,3} =  \av(132, 1234)$, $\mcG_{9,4} =
    \av(132, 4213)$, $\mcG_{9,5} = \av(132, 4123)$, $\mcG_{9,6} = \av(132,
    3124)$ and $\mcG_{9,7} = \av(132, 2134)$ are given by the covers in Table~\ref{tab:93_97}.
    The generating function is equal to $G_{9,1}(x)$.
\end{proposition}

\begin{proof}
    From the covers, we obtain
    $
        G_{9,3}(x) = 1 + x E_{3,1}(x) G_{9,3}(x)
    $,
    and solving gives the claimed equation.
\end{proof}

\begin{table}[ht]
  \centering
  \begin{tabular}{|l|c|}
      \hline
      $\av(132, 1234)$ &
      $
      \onebox \sqcup
      \strule{\struleds}{3/3}[0/1/$\mcE_{3,1}$, 1/2/$\bullet$, 2/0/$\mcG_{9,3}$]
      $\\
      \hline
      $\av(132, 4213)$ &
      $
      \onebox \sqcup
      \strule{\struleds}{3/3}[0/1/$\mcG_{9,4}$, 1/2/$\bullet$, 2/0/$\mcEr_{3,3}$]
      $\\
      \hline
      $\av(132, 4123)$ &
      $
      \onebox \sqcup
      \strule{\struleds}{3/3}[0/1/$\mcG_{9,5}$, 1/2/$\bullet$, 2/0/$\mcE_{3,1}$]
      $\\
      \hline
      $\av(132, 3124)$ &
      $
      \onebox \sqcup
      \strule{\struleds}{3/3}[0/1/$\mcE_{3,2}$, 1/2/$\bullet$, 2/0/$\mcG_{9,6}$]
      $\\
      \hline
      $\av(132, 2134)$ &
      $
      \onebox \sqcup
      \strule{\struleds}{3/3}[0/1/$\mcEr_{3,3}$, 1/2/$\bullet$, 2/0/$\mcG_{9,7}$]
      $\\
      \hline
  \end{tabular}
  \caption{The structure of $\mcG_{9,3}$, $\mcG_{9,4}$, $\mcG_{9,5}$,
           $\mcG_{9,6}$ and $\mcG_{9,7}$}
  \label{tab:93_97}
\end{table}

There are two more bases in this Wilf-class, $\mcG_{9,8} = \av(321, 2341)$ and
$\mcG_{9,9} = \av(321, 3412)$. We have not been able to find a cover for these
permutation classes. The reason for this appears linked to the reason there is
no cover for $\av(123)$. We will revisit these permutation classes in
Section~\ref{sec:further_improvements}, in particular,
Proposition~\ref{prop:321_2341+321_3412}.

\section{Comparison with existing approaches}
  \label{sec:polynomial_bases}

\subsection{Polynomial permutation classes}

A permutation class $\mcC$ is said to be \emph{polynomial} if the number of
length $n$ permutations, $|\mcC_n|$, is given by a polynomial for all
sufficiently large $n$. One of the first general results on permutation classes
by \cite{fibonaccidichotomy} states that if the number of length $n$
permutations in a permutation class is less than the $n$th Fibonacci number then
the permutation class is polynomial. This is known as the Fibonacci dichotomy
and alternative proofs were given by \cite{fibonaccistructure} and
\cite{polynomialconditions}. From the results of \cite{polynomialenumeration}, we
get the following theorem.

\begin{theorem}\label{thm:pegimpliesstruct}
  All polynomial permutation classes have a cover.
\end{theorem}

In order to prove this, we will recall some definitions used by
\cite{polynomialenumeration}. A \emph{peg permutation} is a permutation where
each letter is decorated with a $+$, $-$ or $\circ$, for example, $\rho =
3^{\circ} 1^- 4^{\circ} 2^+$. Let $M_{\rho}$ be the matrix defined by
\[
  M_{i,j} =
  \begin{cases}
    \av(12) \phantom{-} \text{ if } \rho_i = j^+ \\
    \av(21) \phantom{-} \text{ if } \rho_i = j^- \\
    \{ 1 \} \phantom{llllllll} \text{ if } \rho_i = j^\circ
  \end{cases}
\]
then $\grid(\rho) = \grid(M_{\rho})$.

A peg permutation is, therefore, a geometric grid class with monotone intervals
for its matrix entries. We can specify these intervals with vectors on
$\mathbb{N}$, the non-negative integers. We call this a \emph{$\rho$-partition}.
For example, we could write
\[
  6 321 7 45 = 3^\circ 1^- 4^\circ 2^+[\langle 1,3,1,2 \rangle].
\]
Throughout it is insisted that we use vectors that \emph{fill} peg permutations,
meaning that a component of a vector equals $1$ if it corresponds to a $\circ$
and otherwise is at least $2$. For a set of filling vectors $\mcV$, define
\[
  \rho[\mcV] = \{ \rho[v] : v \in \mcV \}.
\]
Given vectors $v$ and $w$ in $\mathbb{N}^m$, then $v$ is contained in $w$ if
$v(i) \leq w(i)$ for all indices $i$. This is a partial order and moreover, if
$v$ is contained in $w$ then for a length $m$ peg permutation, $\rho[v]$ is
contained in $\rho[w]$, assuming this is defined. A set closed downwards under
containment is called a \emph{downset}, and closed upwards an \emph{upset}. The
intersection of a downset and an upset is called a \emph{convex set}.

The set of vectors which fill a given peg permutation $\rho$ forms a convex set.
The downset component of this convex set consists of those vectors which do not
contain an entry larger than 1 corresponding to a dotted entry of $\rho$. The
upset component consists of those vectors which contain the vector $v$ defined
by $v(i) = 1$ if $\rho(i)$ is dotted and $\rho(i) = 2$ if $\rho(i)$ is signed.
As we discussed in Section~\ref{sec:introduction}, all polynomial permutation
classes can be represented by a finite set of peg permutations. In fact, a more
general condition holds.

We can now state the result from \cite[Theorem 1.4]{polynomialenumeration}.

\begin{theorem}[\cite{polynomialenumeration}, Theorem 1.4 and Proposition 2.3]
  \label{thm:heartofpegimpliesstruct}
  For every polynomial permutation class $\mcC$ there is a finite set $H$ of peg
  permutations, each associated with its own convex set $\mcV_\rho$ of filling
  vectors, such that $\mcC$ can be written as the disjoint union
  \[
    \mcC = \bigsqcup_{\rho \in H} \rho[\mcV_\rho].
  \]
\end{theorem}

In \cite[Proposition 2.3]{polynomialenumeration} the authors show that every
permutation which fills a peg permutation $\rho$ has a unique $\rho$-partition.
Together with Theorem~\ref{thm:heartofpegimpliesstruct}, this leads to the
following result.

\begin{theorem}
  A peg permutation $\rho$ and its convex set $\mcV_\rho$ of filling
  vectors is a \struct{} rule.
\end{theorem}

\begin{proof}
  For a peg permutation, $\rho$, $V_\rho$ consists of the vectors with
  $\rho(i) = 1$ when $v(i)$ is dotted and the remaining elements an integer
  greater than $1$. We create $\rho'$ from $\rho$ by replacing an entry $x^+$ in
  $\rho$ with the subsequence $(x-0.2)^\circ (x-0.1)^\circ x^+$ and entries
  $y^-$ with the subsequence $(y+0.2)^\circ (y+0.1)^\circ y^-$ and taking the
  standardization of the underlying permutation. The set $\grid(M_{\rho'})$ is a
  generalized grid class. Moreover, it is a \struct{} rule since every
  permutation which fills $\rho$ has a unique $\rho$-partition.
\end{proof}

Theorem~\ref{thm:pegimpliesstruct} follows as a corollary, as we now have a
finite set of \struct{} rules. Given there exists a cover for all polynomial
permutation classes a natural follow-up question is: for a polynomial
permutation class $\mcC$ is there a bound on the size of the \struct{} rules
required in a cover for $\mcC$?

\subsection{Comparison with other algorithms}

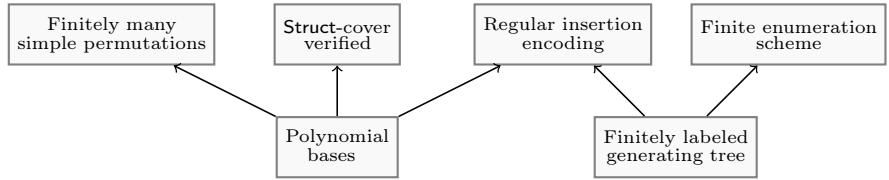
\begin{figure}[h]
 \begin{center}
\begin{tikzpicture}[scale = 0.75, place/.style = {rectangle,draw = black!50,fill = gray!5,thick, minimum size=0.8cm}, auto]

 \node [place] (struct) 	at (2,1)       {$\substack{\text{\struct{}-cover} \\ \text{verified}}$};
 \node [place] (poly) 	    at (2, -1)   {$\substack{\text{Polynomial} \\ \text{bases\phantom{p}}}$};
 \node [place] (gentree) 	at (8,-1)   {$\substack{\text{Finitely labeled} \\ \text{generating tree}}$};
 \node [place] (rie) 	    at (6,1)   {$\substack{\text{Regular insertion} \\ \text{encoding}}$};
 \node [place] (scheme)     at (10,1)  {$\substack{\text{Finite enumeration} \\ \text{scheme}}$};
 \node [place] (simple)     at (-2,1)    {$\substack{\text{Finitely many} \\ \text{simple permutations}}$};

 \draw [->,semithick] (poly) to (struct);
 \draw [->,semithick] (poly) to (rie);
 \draw [->,semithick] (poly) to (simple);
 \draw [->,semithick] (gentree) to (rie);
 \draw [->,semithick] (gentree) to (scheme);

\end{tikzpicture}
 \caption{Comparison of enumeration methods}
 \label{fig:overview}
 \end{center}
\end{figure}

In the remainder of this section, we collect examples of permutation classes
which separate the automatic methods mentioned in the introduction. Regular
insertion encodable permutation classes are always rational. This implies that
$\av(231)$ does not have a regular insertion encoding. As we saw in
Section~\ref{fig:231} this permutation class has a \struct{} cover. It has a
finite enumeration scheme detailed by \cite[p.~7-10]{enumerationschemes} and has
finitely many simple permutations and, therefore, \cite[Theorem
8]{substitutiondecomposition} tells us it can be enumerated using the
substitution decomposition. For $\av(123)$ only finite enumeration schemes have
success, given in plain English by \cite[p.~3-5]{originalenumerationschemes}.

The permutation class $\av(123, 3214)$ has a finitely labeled generating tree
given by \cite[p.~10]{finitelabeled}, but we have not found a \struct{} cover
for it. The permutation class $\av(321, 2143)$, however, does have a cover shown
in Section~\ref{sec:321_2143}, although it can not be enumerated using any of
the other automatic methods.

We have from Theorem~\ref{thm:pegimpliesstruct} that all polynomial permutation
classes have a cover. They are also regular insertion encodable. With a quick check
(for example using the conditions given by \cite[Theorem
1]{polynomialconditions}) we see that $\av(1234, 4231)$ is a polynomial
permutation class, therefore is regular insertion encodable and has a cover. As
was shown by \cite[Proposition 7.1]{enumerationschemes} this permutation class
does not have an enumeration scheme.

The separable permutations $\av(2413,3142)$ have finitely many simple
permutations and so by \cite[Theorem 8]{substitutiondecomposition} they can be
enumerated using the substitution decomposition. This permutation class can not
be enumerated using any of the other methods.

\begin{table}[ht]
\scriptsize
  \centering
  \begin{tabular}{|c|c|c|c|c|c|c|}
      \hline
      basis        & Polynom. & Fin.\ gen.\ tree & Reg.\ ins.\ enc. & \struct{} cover & Fin.\ enum.\ scheme & Fin.\ many simple perms. \\
      \hline
      $132$        &            &                  &                  &   \checkmark   &   \checkmark   &  \checkmark   \\
      \hline
      $123$        &            &                  &                  &        &   \checkmark    &      \\
      \hline
      $123, 3214$  &            & \checkmark     &  \checkmark   &        &   \checkmark   &      \\
      \hline
      $321, 2143$  &            &                  &                  &   \checkmark    &       &     \\
      \hline
      $1234, 4231$ & \checkmark    &         &  \checkmark   &   \checkmark    &        &  \checkmark  \\
      \hline
      $2413, 3142$ &            &                  &                  &        &        &  \checkmark \\
      \hline
  \end{tabular}
  \caption{Bases separating the methods}
  \label{tab:sep_examples}
\end{table}

\section{Patterns of length four}
  \label{sec:large_bases_of_length_four_patterns}

The enumeration and Wilf-classification are known for all permutation classes
with a subset of $\mcS_3$ for a basis. As mentioned, there exists a \struct{}
cover for all of these classes except $\av(123)$, which we revisit in
Section~\ref{sec:further_improvements}.

We will now look at the results of applying the algorithm to bases containing
length four patterns. The total number of bases is $2^{24}$, but of course, it
suffices to look at one basis from each symmetry class. This brings the total
down to $2\,097\,152 \approx 2^{21}$ bases.\footnote{In general one can count subsets of
$S_n$ with respect to the symmetries of the square. This was added to the OEIS
by the authors (\oeis{A277086}). It follows from a simple application of
Burnside's Lemma.} In Section~\ref{sec:polynomial_bases} we showed that all
polynomial permutation classes have a cover. We, therefore, look at the
non-polynomial classes, bringing the total down to $157\,736 \approx 2^{17}$
bases.

As is to be expected \struct{} does better on bases with many patterns. In
Table~\ref{tab:large_bases_all} we break down the computer runs by the size of
the basis and the size of the largest dimension of the \struct{} rules required
for the cover.\footnote{Most of the computation was done on a cluster owned by
Reykjavik University, and the remainder on a cluster owned by the University of
Iceland.} All of the non-polynomial bases were tried with a cover with maximum
rule size of $7 \times 7$. We consider the permutation class failed if \struct{}
does not find a cover within this bound. Of course, there may exist covers using
larger rules.

    \begin{table}[ht]
    \scriptsize
      \centering
      \begin{tabular}{|c|c|c|c|c|c|c|c|c|c|}
          \hline
                &                 &                 &             & \multicolumn{6}{| c |}{successes}\\
          \hline
          $|B|$ & non-symmetric\  & non-polynomial\ & failures \  & $2\times2$ & $3\times3$ & $4\times4$ & $5\times5$ & $6\times6$ & $7\times7$ \\
          \hline
          24 & 1      & 0      & 0    & 0 & 0 & 0 & 0 & 0 & 0 \\
          23 & 7      & 0      & 0    & 0 & 0 & 0 & 0 & 0 & 0 \\
          22 & 56     & 0      & 0    & 0 & 0 & 0 & 0 & 0 & 0 \\
          21 & 317    & 0      & 0    & 0 & 0 & 0 & 0 & 0 & 0 \\
          20 & 1524   & 0      & 0    & 0 & 0 & 0 & 0 & 0 & 0 \\
          19 & 5733   & 1      & 0    & 0 & 1 & 0 & 0 & 0 & 0 \\
          18 & 17728  & 9      & 0    & 0 & 8 & 1 & 0 & 0 & 0 \\
          17 & 44767  & 58     & 0    & 0 & 32 & 26 & 0 & 0 & 0 \\
          16 & 94427  & 285    & 0    & 0 & 75 & 206 & 4 & 0 & 0 \\
          15 & 166786 & 1069   & 0    & 0 & 118 & 901 & 49 & 1 & 0 \\
          14 & 249624 & 3143   & 0    & 0 & 137 & 2620 & 377 & 9 & 0 \\
          13 & 316950 & 7338   & 0    & 0 & 122 & 5118 & 2038 & 60 & 0 \\
          12 & 343424 & 13891  & 0    & 1 & 82 & 6372 & 7163 & 273 & 0 \\
          11 & 316950 & 21451  & 1    & 0 & 36 & 4890 & 15551 & 970 & 3 \\
          10 & 249624 & 27274  & 12   & 0 & 9 & 2285 & 21947 & 2990 & 31 \\
          9  & 166786 & 28391  & 59   & 0 & 1 & 615 & 19672 & 7856 & 188 \\
          8  & 94427  & 24160  & 177  & 6 & 0 & 85 & 9956 & 13051 & 885 \\
          7  & 44767  & 16489  & 708  & 0 & 0 & 10 & 2267 & 10924 & 2580 \\
          6  & 17728  & 8935   & 3249 & 0 & 0 & 2 & 167 & 3668 & 1849 \\
          5  & 5733   & 3716   & 2597 & 0 & 0 & 0 & 7 & 331 & 781 \\
          4  & 1524   & 1187   & 1160 & 3 & 0 & 0 & 1 & 8 & 15 \\
          3  & 317    & 279    & 279  & 0 & 0 & 0 & 0 & 0 & 0 \\
          2  & 56     & 53     & 53   & 0 & 0 & 0 & 0 & 0 & 0 \\
          1  & 7      & 7      & 7    & 0 & 0 & 0 & 0 & 0 & 0 \\
          \hline
      \end{tabular}
      \caption{Data from running \struct{} on every non-polynomial basis with
      length four patterns}
      \label{tab:large_bases_all}
    \end{table}

Judging from the data in the table, one would have hoped to get some successes
for bases with three patterns. However, at this point the memory usage becomes
infeasible for the computers we have access to, routinely exceeding 32GiB of
memory. We, therefore, expect there to be some such permutation classes with
covers consisting of $7\times7$ rules, although we can not find them at this
point.

\subsection{A collection of successes}

The longest basis of length four patterns corresponding to a non-polynomial
permutation class has $19$ patterns
\begin{align*}
\mcH_1 = \av( &1234, 1243, 1324, 1342, 1423, 1432, 2134, 2143, 2314,\\
      &2341, 2413, 2431, 3124, 3142, 3214, 3241, 4123, 4132, 4213).
\end{align*}
\struct{} finds the cover
\begin{equation} \label{eq:largest_non_poly}
    \mcH_1 = \strule{\struleds}{1/1}[0/0/$\mcF$] \sqcup
    \strule{\struleds}{3/2}[0/0/$\bullet$, 1/1/$\bullet$, 2/1/$\bullet$] \sqcup
    \strule{\struleds}{3/3}[0/1/$\bullet$, 1/0/$\bullet$, 2/2/$\bullet$]
\end{equation}
where $\mcF = \av(123, 132, 213)$. We, therefore, see that $\mcH_1$ consists of
permutations avoiding $123, 132, 213$, with the addition of those same
permutations. This can also be observed directly from the basis of $\mcH_1$: it
consists of all length four patterns containing at least one of the patterns in
the basis of $\mcF$.

In the column corresponding to size $2\times2$ rules, we see three successful
bases with $4$ patterns, six bases with $8$ patterns and one basis with $12$
patterns. All of the covers found are similar, e.g., one of the size $8$ bases
is
\[
    \mcH_2 = \av(1243, 1324, 2143, 2314, 3142, 3214, 4132, 4213)
\]
and has the cover
\begin{equation} \label{eq:length8_small}
    \mcH_2 = \onebox \sqcup
    \strule{\struleds}{2/1}[0/0/$\bullet$, 1/0/$\mcE_{3,3}^r$]
\end{equation}
where $\mcE_{3,3}^r = \av(132, 213)$. The enumeration is therefore given by
$n 2^{n-2}$. The covers for the three successful bases with $4$ patterns are
essentially the same as some of the structures discussed in \cite{bruner}.

There are thirteen conjectures given by \cite{kuszmaul} about the enumeration of
bases containing many length four patterns, of which ten have a regular
insertion encoding. \struct{} found covers for these as well as one of the
remaining three:
\[
    \mcH_3 = \av(2431, 2143, 3142, 4132, 1432, 1342, 1324, 1423, 1243)
\]
\struct{} found the cover
\begin{equation} \label{eq:Kuz2}
    \mcH_3 = \strule{\struleds}{1/1}[0/0/$\mcA^r$] \sqcup
    \strule{\struleds}{4/4}[0/1/$\bullet$, 1/3/$\bullet$, 2/0/$\mcA^r$, 3/2/$\bullet$]
\end{equation}
where $\mcA^r = \av(132)$. This leads to the equation
$H_3(x) = A(x) + x^3 A(x)$ (\oeis{A071742}), where $A(x)$ is
the generating function for the Catalan numbers.

Although it is not a focus of this paper we show how this cover can be used to
enumerate the class with respect to some permutation statistics. If we let
$y$ track the number of left-to-right-minima,
\[
  H_3(x,y) = A(x,y) + x^3 y A(x,y).
\]
Then if we let $z$ track the number of inversions then
\[
  H_3(x,z) = A(x,z) + x^3 z A(xz^2,z).
\]
Finally, if we let $w$ track the number of descents then
\[
  H_3(x,w) = A(x,w) + x^3 A(x,w).
\]

To highlight some interesting structures found by \struct{} we end this section
by showing the covers found for three permutation classes with a basis
consisting of four length four patterns. For the permutation class $\mcH_4 =
\av(1324, 1342, 3124, 3142)$ the cover found was
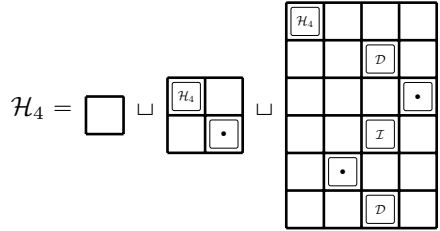
\[
  \mcH_4 = \onebox \sqcup
  \strule{\struleds}{2/2}[1/0/$\bullet$, 0/1/$\mcH_4$] \sqcup
  \strule{\struleds}{4/6}[0/5/$\mcH_4$, 1/1/$\bullet$, 2/0/$\decr$, 2/2/$\incr$, 2/4/$\decr$, 3/3/$\bullet$]
\]
which leads to the equation
$
  H_4(x) = 1 + x H_4(x) + \frac{x^2 H_4(x)}{1-3x}
$
and upon solving gives
$
  H_4(x) = \frac{1 - 3x}{1 - 4x+ 2x^2}
$
(\oeis{A006012}, shifted).

The permutation class $\mcH_5 = \av(1342, 2314, 2413, 3142)$ was enumerated by
\cite[Theorems 18 and 19]{wreathproduct} using methods with the wreath product.
\struct{} finds the cover
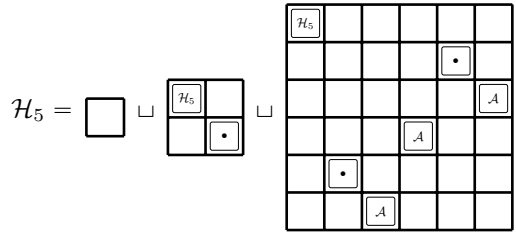
\[
  \mcH_5 = \onebox \sqcup
  \strule{\struleds}{2/2}[0/1/$\mcH_5$,1/0/$\bullet$] \sqcup
  \strule{\struleds}{6/6}[0/5/$\mcH_5$, 1/1/$\bullet$, 2/0/$\mcA$, 3/2/$\mcA$, 4/4/$\bullet$, 5/3/$\mcA$]
\]
where $\mcA = \av(231)$. This leads to the equation
$
  H_5(x) = 1 + xH_5(x) + x^2 H_5(x) A(x)^3
$
and to the same function first found by \cite{wreathproduct} (\oeis{A078483}).

Choosing a permutation at random from a permutation class is, in general, a very
difficult problem. For example, the problem of what does a typical permutation
in a class look like was initiated by \cite{randominitiated}, and, for example,
further studied by \cite{randomminer} and \cite{randomsample}.

We do not address this problem in any generality here but for the cover for
$\mcH_5$ above we can use the following method: first, solve the functional
equation for $H_5(x)$ above to obtain an enumeration of all the permutations in
the class. Then find the enumeration of the permutations in the grid classes of
the individual rules. These are, ignoring the empty rule, $[x^n] x H_5(x)$ and
$[x^n] x^2 A(x)^3 H_5(x)$. These give us weights for a probability distribution
between the rules. Then we need to be able to choose a permutation uniformly at
random from a given rule. We only look at the larger rule here. We use variables
$a$, $b$, $c$, $d$, to track the number of points in each of the blocks
containing a permutation class, read from left to right. This gives the
generating function $x^2H_5(ax)A(bx)A(cx)A(dx)$. The probability of choosing a
permutation with $i$, $j$, $k$, $\ell$ in each of the blocks is now seen to be \[
\frac{[x^n a^i b^j c^k d^\ell] x^2H_5(ax)A(bx)A(cx)A(dx)}{[x^n] x^2 A(x)^3
H_5(x)}. \] We then need to recursively apply this method to choose a length $i$
permutation from $\mcH_5$ and choose random permutations of lengths $j$, $k$,
$\ell$ from $\mcA$. This method was formalized by \cite{randomsamplemethod}.

The enumeration of the permutation class $\mcH_6 = \av(1342,2413,3124,3142)$ is
not on the OEIS. The cover found by \struct{} is
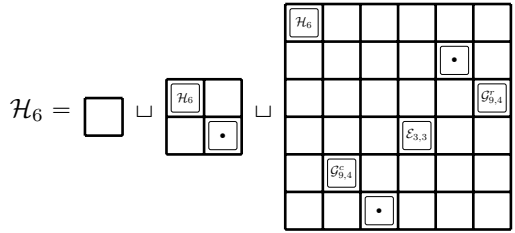
\[
  \mcH_6 = \onebox \sqcup
  \strule{\struleds}{2/2}[0/1/$\mcH_6$, 1/0/$\bullet$] \sqcup
  \strule{\struleds}{6/6}[0/5/$\mcH_6$, 1/1/$\mcG_{9,4}^c$, 2/0/$\bullet$, 3/2/$\mcE_{3,3}$, 4/4/$\bullet$, 5/3/$\mcG_{9,4}^r$]
\]
where $\mcG_{9,4}^c = \av(312, 1342)$ and $\mcG_{9,4}^r = \av(231, 3124)$.  This
can be easily verified and leads to the equation
$
  H_6(x) = 1 + xH_6(x) + x E_{3,1}(x) G_{9,1}(x)^2.
$
Solving gives $H_6(x) = \frac{1 - 6x + 11x^2 - 6x^3 + x^4}{1 - 7x + 16x^2 -
14x^3 + 5x^4 - x^5}$.

This cover is typical of many found by \struct{}, in that it places into the
rule a permutation class of the form $\av(P, Q)$ where $P$ is a non-empty subset
of $\mcS_3$ and $Q$ is a subset of $\mcS_4$.\footnote{Up to symmetries there are
$14\,181$ such bases.} It is known that all such classes have a rational
generating function, see \cite{rational321} and
\cite{substitutiondecomposition}, but they have not been computed. We believe
that these should all be computed as a first step in attempting to enumerate all
permutation classes containing length four patterns.

\subsection{Discussion of the failures}
The permutation class with the largest basis that \struct{} fails to find a cover for
using $7\times7$ rules, has $11$ patterns,
namely
\[
    \av(1234, 1243, 1324, 1342, 1423, 1432, 2134, 2143, 2314, 2341, 3214).
\]
It can, however, be done with the same methods as are needed for the \struct{}
failures we discuss in Section~\ref{sec:further_improvements}. This leads to the
generating function $\frac{1 - x + x^3}{1 - 2x - x^3}$.

After the first failure, the failure rate stays below $1\%$ for bases with at
least $8$ patterns, rising to $5\%$ for $7$ patterns, $37\%$ for $6$ patterns,
$70\%$ for $5$ patterns, $98\%$ for $4$ patterns, until every basis with $3$ or
fewer patterns fails.



\section{Further improvements}
  \label{sec:further_improvements}

Given a conjectured cover from \struct{} for a permutation class $\mcC$, it is
natural to want to check that this is, in fact, a cover for $\mcC$. Given an
individual \struct{} rule it can be verified automatically if it generates a
subset of a permutation class $\mcC$: If there is a permutation $\pi$ in the
\struct{} rule $\mcR$ which contains a basis element of length $k$, remove all the
non-points from the gridding of the permutation so that the permutation still
contains an occurrence of the basis element. This reduced permutation has at
most $\ell + k$ points. Therefore if $\mcR_{\leq \ell + k} \subseteq \mcC$,
where $k$ is the length of the longest basis element, then $\mcR \subseteq
\mcC$. This brings us closer to an automatic check if a cover found by \struct{} is
guaranteed to be a subset of the permutation class $\mcC$. What remains is there
might still be an overlap between the permutations generated by two rules. Below
we sketch how the steps one usual performs to verify the cover by hand, can be
turned into an algorithm.

\subsection{Proof Trees}

To motivate the section we revisit the proof of
Proposition~\ref{prop:132_4312+132_4231}, concerning the permutation class
$\mcG_{7,2} = \av(132, 4231)$, and turn the verification of the cover into a
``proof tree'' in Figure~\ref{fig:132_4231}.

We start with a root vertex representing $\mcG_{7,2}$. The empty permutation is
in $\mcG_{7,2}$, and we add a left child of the root indicating this. All other
permutations have a topmost point and we add the right child to represent them.
This is analogous as to saying a non-empty permutation can be written as $\alpha n
\beta$. In our tree, we observe that the right part must avoid $231$, and so be
in $\av(132,231)$. We continue with our argument, for permutations where $\beta$
is empty, $\alpha$ can be any permutation in $\mcG_{7,2}$ so we indicate this by
adding a left child of the node we are on. Otherwise, $\beta$ has a leftmost
point, and we add a right child to represent them. We observe that the points to
the right and above the leftmost point must be in $\av(21)$, and to the right
and below must be in $\av(132,231)$. We have now reached the left tree in
Figure~\ref{fig:132_4231}.

\begin{figure}[htbp]
    \centering
        \centering
        \begin{tikzpicture}[scale = 0.75, vertex/.style = {}, auto]

         \node [vertex] (root) 	 at (0,0)     {\strule{\struleds}{1/1}[0/0/$\mcG_{7,2}$]};
         \node [vertex] (empty)  at (-1,-2)   {$\onebox$};
         \node [vertex] (nempty) at (1,-2)    {\strule{\struleds}{3/2}[0/0/$\mcG_{7,2}$, 2/0/$\mcEi_{3,2}$, 1/1/$\bullet$]};
         \node [vertex] (nemptyl) at (-1,-5)    {\strule{\struleds}{2/2}[0/0/$\mcG_{7,2}$, 1/1/$\bullet$]};
         \node [vertex] (nemptyr) at (2,-5)    {\strule{\struleds}{4/4}[0/2/$\mcE_{3,2}$, 1/3/$\bullet$, 2/1/$\bullet$, 3/0/$\mcEi_{3,2}$, 3/2/$\mcI$]};

         \draw [-,semithick] (root) to (empty);
         \draw [-,semithick] (root) to (nempty);
         \draw [-,semithick] (nempty) to (nemptyl);
         \draw [-,semithick] (nempty) to (nemptyr);

         \begin{scope}[xshift=10cm]
             \node [vertex] (root) 	 at (0,0)     {\strule{\struleds}{1/1}[0/0/$\mcG_{7,2}$]};
             \node [vertex] (empty)  at (-1,-2)   {$\onebox$};
             \node [vertex] (nempty) at (1,-2)    {\strule{\struleds}{3/2}[0/0/$\mcG_{7,2}$, 2/0/$\mcEi_{3,2}$, 1/1/$\bullet$]};
             \node [vertex] (nemptyl) at (-1,-5)    {\strule{\struleds}{2/2}[0/0/$\mcG_{7,2}$, 1/1/$\bullet$]};
             \node [vertex] (nemptyr) at (2,-5)    {\strule{\struleds}{5/5}[0/3/$\mcE_{3,2}$, 1/4/$\bullet$, 2/1/$\bullet$, 3/0/$\mcEi_{3,2}$, 4/2/$\mcI$]};

             \draw [-,semithick] (root) to (empty);
             \draw [-,semithick] (root) to (nempty);
             \draw [-,semithick] (nempty) to (nemptyl);
             \draw [-,semithick] (nempty) to (nemptyr);
         \end{scope}
        \end{tikzpicture}
    \caption{The structure of $\av(132, 4231)$.}
    \label{fig:132_4231}
\end{figure}
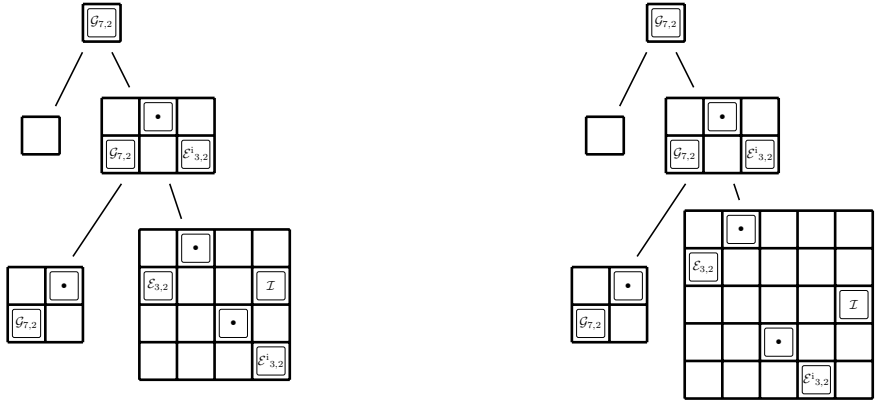

\noindent
The only ``proof strategy'' (PS) used so far is
\begin{enumerate}
  \item[(PSa)]
   Let $v$ be a leaf with some cell $c$ containing a non-empty permutation
   class. Create two children $v^\ell$, $v^r$ of $v$ where in $v^\ell$ the empty
   permutation was chosen in $c$; and in $v^r$ an extreme point (right, top,
   left or bottom) has been inserted. Infer as much information as possible
   about the newly created generalized grid class rule in $v^r$.
\end{enumerate}

\noindent
We continue with our example, and observe that there can not be a non-inversion
between the cells containing $\mcG_{7,2}$ and $\mcI$, because an occurrence of
$132$ would be formed. We can, therefore, draw a horizontal line between these
cells and separate them. We also observe there can not be an inversion between
the cells containing $\mcI$ and $\mcEi_{3,2}$, otherwise, an occurrence of $4231$
would be formed. We can, therefore, draw a vertical line between these cells and
separate them. We have now reached the right tree in Figure~\ref{fig:132_4231}.

\noindent
The proof strategy we used was
\begin{enumerate}
	\item[(PSb${}^\textrm{ri}$)] Let $v$ be a leaf with only two non-empty cells
    $s$ and $t$ in a particular row. If an inversion crossing between $s$ and
    $t$ creates a pattern in the basis, draw a horizontal line splitting the
    row. Let $s_\uparrow$, $s_\downarrow$, and $t_\uparrow$, $t_\downarrow$, be
    the cells $s$ and $t$ split into, then $s_\uparrow = t_\downarrow = {\epsilon}$,
    and infer as much information as possible about the newly created cells.
\end{enumerate}

\noindent
There are three other versions of this proof strategy, about cells in the same
column and insertion of a non-inversion. We collectively call them (PSb).\\

Applying the proof strategies (PSa) and (PSb) creates trees with the
property that the disjoint union of children is a superset of their parent.
We will refer to this as the \emph{superset property}.

From the superset property, it follows that at each stage the leaves form a
superset, in the form of a disjoint union, of the root. Leaves that are subsets
of the root, i.e., avoid the patterns in the basis will be called
\emph{verified}. If every leaf is verified, then the root is a disjoint union of
the sets represented by the leaves.

For all the covers seen in this paper so far, it is possible to draw a similar
proof tree using (PSa) and (PSb). We will now look at bases \struct{} did not
find a cover for. In order to find proof trees, we need a new proof strategy.

\begin{enumerate}
	\item[(PSc)] Let $\av(B)$ be the set under consideration. Let $v$ be a
    non-verified leaf in a tree. Let $v'$ be a subset of cells in
    $v$ satisfying the following two conditions:
    \begin{enumerate}
        \item[(1)] When every other cell is deleted, what remains is an
        ancestor vertex $u$ of $v$.
        \item[(2)] Let $c$ be a placement of points in $u$ resulting in a permutation
        in $\av(B)$ and let $c'$ be the corresponding placement of points in
        $v'$. Then any completion of the placement $c'$ to a placement on $v$
        results in a permutation in $\av(B)$.
    \end{enumerate}
    We declare the leaf as verified and modify the permutations represented by
    $v$ by using only placements of the form defined in condition (2).
\end{enumerate}

\noindent
When (PSc) is applied to a subset $v'$ of a vertex $v$ we add a dotted edge from
the ancestor vertex $u$ to $v$. This is not a part of the tree, just to keep
track of where (PSc) has been applied. With our new proof strategy, we now
revisit some bases that \struct{} did not find a cover for.

\subsection{One pattern of length three and one of length four - revisited}

In Section~\ref{sec:large_bases_of_length_four_patterns} we found covers for
almost all of the bases with one length three pattern and one length four
pattern. The covers we found can all be converted into proof trees. There
were two bases that \struct{} could not find a cover for. There does exist a
proof tree for these bases. Both arguments are similar and we will only provide
a proof for the structure found in the first case.

\begin{proposition}[\cite{generatingtrees}]
    \label{prop:321_2341+321_3412} The structures of the permutation classes
    $\mcG_{9,8} = \av(321, 2341)$ and $\mcG_{9,9} = \av(321, 3412)$ are given by
    the proof trees in Figure~\ref{fig:321_2341+321_3412}. Their generating
    functions are equal to $G_{9,1}(x)$.
\end{proposition}

\begin{figure}[hbtp]
    \centering
      \centering
      \begin{tikzpicture}[scale = 0.75, vertex/.style = {}, auto]

        \node [vertex] (root) 	 at (0,0)     {\strule{\struleds}{1/1}[0/0/$\mcG_{9,8}$]};
        \node [vertex] (empty)  at (-1,-2)   {\onebox};
        \node [vertex] (nempty) at (1,-2)    {\strule{\struleds}{3/2}[2/0/$\mcI$, 0/0/$\mcG_{9,8}$, 1/1/$\bullet$]};
        \node [vertex] (nemptyl) at (-1,-5)    {\strule{\struleds}{2/2}[0/0/$\mcG_{9,8}$, 1/1/$\bullet$]};
        \node [vertex] (nemptyr) at (2,-5)    {\strule{\struleds}{4/4}[3/2/$\mcI$, 2/1/$\bullet$, 1/3/$\bullet$, 0/0/$\mcG_{9,8}$, 0/2/$\mcI$]};
        \node [vertex] (nemptyrl) at (0,-9)    {\strule{\struleds}{4/4}[3/2/$\mcI$, 2/1/$\bullet$, 1/3/$\bullet$, 0/0/$\mcG_{9,8}$]};
        \node [vertex] (nemptyrr) at (5,-9)    {\strule{\struleds}{7/6}[5/2/$\mcI$, 6/4/$\mcI$, 4/1/$\bullet$, 3/5/$\bullet$, 2/0/$\mcI$, 1/3/$\bullet$, 0/0/$\mcG_{9,8}$]};

       \draw [-,semithick] (root) to (empty);
       \draw [-,semithick] (root) to (nempty);
       \draw [-,semithick] (nempty) to (nemptyl);
       \draw [-,semithick] (nempty) to (nemptyr);
       \draw [-,semithick] (nemptyr) to (nemptyrl);
       \draw [-,semithick] (nemptyr) to (nemptyrr);

       \path[->,dashed,semithick]  (nempty)  edge   [bend left]   node {(PSc)} (nemptyrr);

       \begin{scope}[xshift=10cm]
           \node [vertex] (root) 	 at (0,0)     {\strule{\struleds}{1/1}[0/0/$\mcG_{9,9}$]};
           \node [vertex] (empty)  at (-1,-2)   {\onebox};
           \node [vertex] (nempty) at (1,-2)    {\strule{\struleds}{3/2}[2/0/$\mcI$, 0/0/$\mcG_{9,9}$, 1/1/$\bullet$]};
           \node [vertex] (nemptyl) at (-1,-5)    {\strule{\struleds}{2/2}[0/0/$\mcG_{9,9}$, 1/1/$\bullet$]};
           \node [vertex] (nemptyr) at (2,-5)    {\strule{\struleds}{4/4}[3/1/$\bullet$, 2/0/$\mcI$, 1/3/$\bullet$, 0/0/$\mcG_{9,9}$, 0/2/$\mcI$]};
           \node [vertex] (nemptyrl) at (0,-9)    {\strule{\struleds}{4/3}[3/1/$\bullet$, 2/0/$\mcI$, 1/2/$\bullet$, 0/0/$\mcG_{9,9}$]};
           \node [vertex] (nemptyrr) at (5,-9)    {\strule{\struleds}{6/5}[5/1/$\bullet$, 4/4/$\bullet$, 3/3/$\mcI$, 2/0/$\mcI$, 1/2/$\bullet$, 0/0/$\mcG_{9,9}$]};

          \draw [-,semithick] (root) to (empty);
          \draw [-,semithick] (root) to (nempty);
          \draw [-,semithick] (nempty) to (nemptyl);
          \draw [-,semithick] (nempty) to (nemptyr);
          \draw [-,semithick] (nemptyr) to (nemptyrl);
          \draw [-,semithick] (nemptyr) to (nemptyrr);

          \path[->,dashed,semithick]  (nempty)  edge   [bend left]   node {(PSc)} (nemptyrr);
          \path[->,dashed,semithick]  (nempty)  edge   [bend right=55]   node [pos=0.8] {(PSc)} (nemptyrl);
       \end{scope}

        \end{tikzpicture}

      %
      %
      %
      %
      %
    \caption{The structures of $\mcG_{9,8} = \av(321, 2341)$ and $\mcG_{9,9} = \av(321, 3412)$.}
    \label{fig:321_2341+321_3412}
\end{figure}
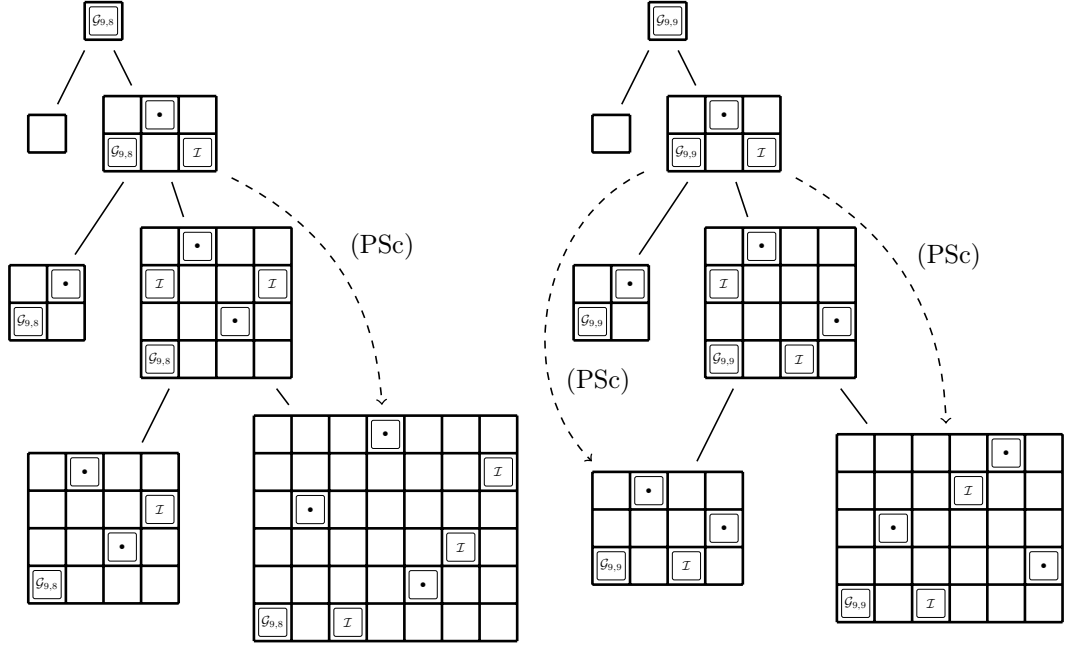

\begin{proof}
  The root node is $\mcG_{9,8}$. We first use (PSa) on the root node, choosing
  the topmost point. The left child is verified. On the right child, we use (PSa)
  on the cell containing $\mcI$, choosing the leftmost point. The left child is
  verified. On the right child, we use (PSa) on the node containing $\mcI$ in the
  far left column, choosing the rightmost. The left child is verified. On the
  right child, we use (PSb) on the two cells in the rightmost column containing
  $\mcI$. We then use (PSc) from the right child of the root vertex to the
  current node. All leaves are now verified. We now have the left proof tree in
  Figure~\ref{fig:321_2341+321_3412}. From this proof tree, we obtain
  \[
  G_{9,8}(x) = 1 + x G_{9,8}(x) + x^2 D(x) G_{9,8}(x) + x^2 D(x)^2
  (G_{9,8}(x)-1).
  \]
  Solving gives the claimed equation. From the proof tree on the right in
  Figure~\ref{fig:321_2341+321_3412}, we obtain
  \[
  G_{9,9}(x) = 1 + x G_{9,9}(x) + x ( G_{9,9}(x) - 1 ) + x^2 D(x)
  (G_{9,9}(x)-1).
  \]
  Solving gives the claimed equation.
\end{proof}

Therefore with the addition of (PSc), all bases with one length three and one
length four pattern have a proof tree.

\subsection{Patterns of length three - revisited}

In Section~\ref{sec:patterns_of_length_three} we found covers for all bases with
patterns of length three, except $\{123\}$. There does exist a proof tree for
$\av(123)$.

\begin{proposition}
  \label{prop:123}
  The structure of the permutation class $\mcA_2 = \av(123)$ is given by the
  proof tree in Figure~\ref{fig:123}.
\end{proposition}

\begin{figure}[hbt]

        \centering
        \begin{tikzpicture}[scale = 0.75, vertex/.style = {}, auto]

          \node [vertex] (root) 	 at (0,0)     {\strule{\struleds}{1/1}[0/0/$\mcA_2$]};
          \node [vertex] (empty)  at (-1,-2)   {\onebox};
          \node [vertex] (nempty) at (1,-2)    {\strule{\struleds}{3/2}[0/1/$\mcA_2$, 1/0/$\bullet$, 2/1/$\mcD$]};
          \node [vertex] (nemptyl) at (-1,-5)    {\strule{\struleds}{2/2}[0/1/$\mcA_2$, 1/0/$\bullet$]};
          \node [vertex] (nemptyr) at (2,-5)    {\strule{\struleds}{4/4}[0/1/$\mcD$, 0/3/$\mcA_2$, 1/0/$\bullet$, 2/2/$\bullet$, 3/1/$\mcD$]};
          \node [vertex] (nemptyrl) at (0,-9)    {\strule{\struleds}{4/4}[0/3/$\mcA_2$, 1/0/$\bullet$, 2/2/$\bullet$, 3/1/$\mcD$]};
          \node [vertex] (nemptyrr) at (5,-9)    {\strule{\struleds}{7/6}[0/5/$\mcA_2$, 1/2/$\bullet$, 2/1/$\mcD$, 2/5/$\mcD$, 3/0/$\bullet$, 4/4/$\bullet$, 5/3/$\mcD$, 6/1/$\mcD$]};

         \draw [-,semithick] (root) to (empty);
         \draw [-,semithick] (root) to (nempty);
         \draw [-,semithick] (nempty) to (nemptyl);
         \draw [-,semithick] (nempty) to (nemptyr);
         \draw [-,semithick] (nemptyr) to (nemptyrl);
         \draw [-,semithick] (nemptyr) to (nemptyrr);

         \path[->,dashed,semithick]  (nempty)  edge   [bend left]   node {(PSc)} (nemptyrr);

       \end{tikzpicture}
    \caption{The structure of $\mcA_2 = \av(123)$.}
    \label{fig:123}
\end{figure}
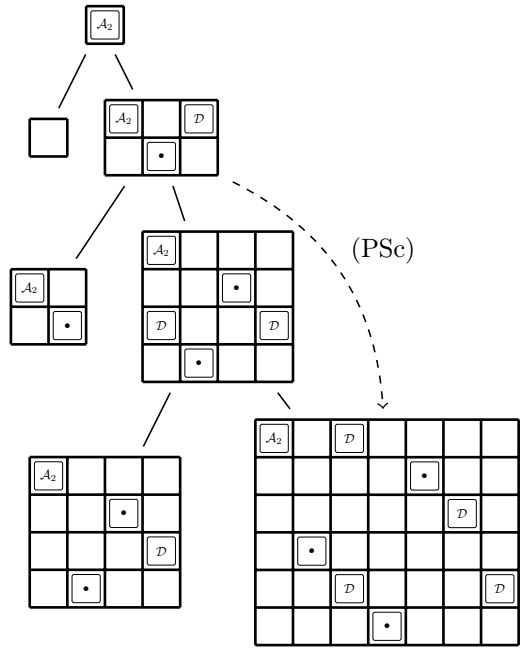

We omit the proof as it follows similar arguments as in the previous
proposition. From the proof tree in Figure~\ref{fig:123}, it is easy to
derive a recurrence for enumerating $\av(123)$. It is not trivial to show that
this enumeration is equal to the Catalan numbers. We are therefore still lacking
more powerful algorithms for a fully automatic Wilf-classification of all bases
in $\mcS_3$.

\section{Conclusion}
  \label{sec:conclusion}

It is decidable whether a basis defines a polynomial permutation class. The same is
true for regular insertion encodable permutation classes. Given that polynomial
permutation classes have covers, it is natural to ask the following.

\begin{question}
  Given a basis $B$, is it decidable whether a cover exists for the permutation
  class $\av(B)$? Moreover, is there a bound on the size of the rules required?
\end{question}

\noindent
There is no known bound on the size of the peg permutations (and therefore
\struct{} rules) for polynomial permutation classes, suggesting that this
question might be hard to answer. In Section~\ref{sec:further_improvements} we
showed that it is easy to verify if a \struct{} rule is a subset of a
permutation class $\mcC$.

\begin{question}
  Can you verify when two rules are disjoint?
\end{question}

Given this, we would then be able to show that a \struct{} cover is a subset of
the permutation class. Of course, the goal is to answer the following.

\begin{question}
  Given a finite set of \struct{} rules, is it decidable if these form a cover
  for a permutation class?
\end{question}

This is a natural question to ask, and the proof trees discussed in
Section~\ref{sec:further_improvements} go some way to answering this, however, it
is not clear that all covers can be made into a proof tree argument.

  \bibliographystyle{plainnat}
  \bibliography{references}

\end{document}